\documentclass[11pt]{amsart}

\usepackage{amssymb, bm}
\usepackage[all]{xy}

\numberwithin{equation}{section}

\newtheorem{theorem}{Theorem}[section]
\newtheorem{mainthm}{Main Theorem}

\newtheorem{corollary}[theorem]{Corollary}
\newtheorem{proposition}[theorem]{Proposition}
\newtheorem{definition}[theorem]{Definition}
\newtheorem{lemma}[theorem]{Lemma}

\newtheorem{remark}[theorem]{Remark}
\newtheorem{example}[theorem]{Example}

\newtheorem{question}[theorem]{Question}
\newtheorem{acknowledgements}{Acknowledgements}

\def\B{{\mathcal{B}}}
\def\E{{\mathcal{E}}}
\def\F{{\mathcal{F}}}
\def\G{{\mathcal{G}}}

\def\X{{\mathfrak{X}}}

\def\hom{{\mathcal{H}om}}

\def\p{{\mathfrak{p}}}
\def\PP{{\mathbb{P}}}
\def\Q{{\mathbb{Q}}}

\def\Z{{\mathbb{Z}}}

\def\Ext{{\mathrm{Ext}}}
\def\Hom{{\mathrm{Hom}}}

\def\Spec{{\mathrm{Spec\; }}}
\def\Proj{{\mathrm{Proj\; }}}
\def\Pic{{\mathrm{Pic\; }}}
\def\id{\mathop{\mathrm{id}}\nolimits}

\def\rk{\mathop{\mathrm{rk}}\nolimits}

\def\mod{\mathop{\mathrm{mod}}\nolimits}

\def\Sch{\mathop{\mathrm{Sch}}\nolimits}

\def\Lis{\text{Lis-\'Et}}
\def\0{{\mathbf{0}}}
\def\1{{\mathbf{1}}}
\def\lra{\longrightarrow}
\def\mbi#1{\boldsymbol{#1}}
\def\eqref#1{(\ref{#1})}
\def\Coh{\mathop{\mathrm{Coh}}\nolimits}

\begin{document}

\title{The FFRT property of two-dimensional normal 
graded rings and orbifold curves
}
\author[N. Hara]{Nobuo Hara}
\author[R. Ohkawa]{Ryo Ohkawa}
\address[N. Hara]{
Tokyo University of Agriculture and Technology, 2--24--16 Nakacho, Koganei, 
Tokyo 184--8588, Japan}
\email{nhara@cc.tuat.ac.jp}
\address[R. Ohkawa]{Waseda Research Institute for Science and Engineering, Waseda University, 3--4--1 Okubo, Shinjuku-ku, Tokyo 169--8555, Japan} 
\email{ohkawa.ryo@aoni.waseda.jp}

\maketitle

\begin{abstract}
We study the finite $F$-representation type (abbr.\ FFRT) 
property of a two-dimensional normal graded ring 
$R$ in characteristic $p>0$, using notions from the theory of algebraic stacks. Given a graded ring $R$, 
we consider an orbifold curve $\mathfrak C$, which is a root stack over the smooth curve $C=\Proj R$, 
such that $R$ is the section ring associated with a line bundle $L$ on $\mathfrak C$. The FFRT property 
of $R$ is then rephrased with respect to the Frobenius push-forwards $F^e_*(L^i)$ on the orbifold curve 
$\mathfrak C$. 
As a result, we see that if the singularity of $R$ is not log terminal, then $R$ has FFRT only in exceptional 
cases where the characteristic $p$ divides a weight of $\mathfrak C$. \\

\end{abstract}


\section{Introduction}

The notion of finite $F$-representation type for a ring $R$ of characteristic $p>0$ was introduced by Smith and Van den Bergh \cite{SVdB}. 
Its definition requires the technical assumption that $R$ is $F$-finite and either a complete 
local domain or an Noetherian $\mathbb{N}$-graded domain. For each $e\in\mathbb{N}$ we identify the 
ring $R^{1/p^e}$ of $p^e$-th roots of $R$ with the $e$-times iterated Frobenius push-forward of the 
structure sheaf of $\Spec R$. We say that $R$ has \emph{finite $F$-representation type} (FFRT for short), if 
the set of isomorphism classes of indecomposable modules appearing as a direct summand of $R^{1/p^e}$ 
as an $R$-module for some $e$, is finite. 

If $R$ is a regular local ring or a polynomial ring, then it has FFRT, since $R^{1/p^e}$ is a free $R$-module. 
It is shown in [SVdB] that a finite direct summand of a ring of FFRT also has FFRT. 
In particular, $R$ has FFRT if $R$ has a tame quotient singularity such as the invariant subring of a finite group of order not divisible by $p$ acting on a 
regular local ring.
Also, it is known that a Frobenius sandwich singularity such as 
$R = k[x,y,z]/\langle z^p-f(x,y)\rangle$ has FFRT (cf.\ \cite{Sh}). On the other hand, simple elliptic 
singularities, and more generally, cone singularities over a smooth curve of genus $g\ge 1$, are known not 
to have FFRT \cite{SVdB}. 

In this paper, we will explore the FFRT property for normal surface singularities with $k^*$-action, in 
other words, a two-dimensional normal graded ring $R$ over a field $k=R_0$. 
In this case, $C=\Proj R$ is a 
smooth curve and there is an ample $\Q$-divisor on $C$ such that 
$R \cong R(C,D) = \bigoplus_{m \ge0} H^0(C,\mathcal{O}_C(\lfloor m D\rfloor))$. 
The result for cone singularities implies that we cannot expect for the FFRT property unless $C \cong \PP^1$.
Thus the critical case is when $R=R(\PP^1,D)$ and it is not log terminal. As far as the authors are 
aware, the FFRT property of such an $R$ is wide open. Specifically we aim to answer the following: 
\begin{question}[Holger Brenner, 2007]
Does the ring $R=k[x,y,z]/\langle x^2+y^3+z^7\rangle$ have FFRT?
\end{question}
\noindent
We note that the ring $R$ in Brenner's question is 
given by a $\Q$-divisor $D=\frac12(\infty)-\frac13(0)-\frac17(1)$ on $\PP^1$. If $p=2,3,7$, it is a 
Frobenius sandwich and so has FFRT \cite{Sh}. Our main result in this paper is the following, which 
implies that $R=k[x,y,z]/\langle x^2+y^3+z^7\rangle$ does not have FFRT unless $p=2,3,7$. 

\begin{mainthm}[ = Theorem \ref{MainThm}]
Let $R=R(\PP^1,D)$ for an ample $\Q$-divisor $D$ on $\PP^1$. If $R$ does not have a log terminal singularity 
and if the characteristic $p$ does not divide any denominator appearing in the rational coefficients of 
$D$, then $R$ does not have FFRT.
\end{mainthm}
 
In the study of the structure of $R^{1/p^e}$ we have difficulty with the non-integral rational coefficients 
of $D$. 
To overcome this difficulty we will introduce an orbifold curve (called a weighted projective line when 
$C \cong \PP^1$), which is a sort of Deligne-Mumford stack 
$\mathfrak{C}$ with coarse moduli map $\pi \colon \mathfrak{C} \to C$. This is the ``minimal covering" 
of $C=\Proj R$ on which $D$ becomes integral. Thus it serves as a very useful tool to deal with rational 
coefficients divisors (cf.\ \cite{MO}). 
The FFRT property of $R$ is then rephrased in terms of a global analogue 
of FFRT property for a pair $(\mathfrak{C},L)$ associated with the line bundle $L=\mathcal{O}_\mathfrak{C}(\pi^*D)$. 
In particular, by studying the structure of the Frobenius push-forwards $F^e_*\mathcal{O}_\mathfrak{C}$ on the 
orbifold curve $\mathfrak{C}$, we prove our results. 

Let us give an overview of the proof of our main theorem in a bit more detail. 
The assumption that the singularity of $R$ is not log 
terminal is equivalent to the condition $\delta_{\mathfrak{C}} \ge 0$, where $\delta_{\mathfrak{C}}$ is the 
degree of the canonical bundle on $\mathfrak{C}$. 
When $\delta_{\mathfrak{C}}=0$, we have an \'etale 
covering $\varphi \colon E \to \mathfrak{C}$ from an elliptic curve $E$, via which the Frobenius push-forward 
$F^e_*\mathcal{O}_\mathfrak{C}$ is related to that on $E$. Since the structure of $F^e_*\mathcal{O}_E$ on an elliptic curve 
$E$ is well-understood \cite{A}, \cite{Od}, we can deduce the result for $F^e_*\mathcal{O}_\mathfrak{C}$, whose structure 
differs according to whether $E$ is ordinary or supersingular. 
When $\delta_{\mathfrak{C}} > 0$, we prove 
the stability of $F^e_*\mathcal{O}_\mathfrak{C}$ by the method of \cite{KS}, \cite{Su} for non-orbifold curves of genus 
$>1$, from which it follows that $F^e_*\mathcal{O}_\mathfrak{C}$ is indecomposable. 

This paper is organized as follows. In Section 2 we review some fundamental facts on normal graded rings  
and root stacks. 
In Section 3 we 
rephrase the FFRT property of $R=R(C,D)$ in terms of a global FFRT property on the orbifold curve $\mathfrak{C}$ constructed from $(C,D)$. 
Sections 4 and 5 are
devoted to the study of weighted projective lines with $\delta_\mathfrak{C} \le 0$. 
In Section 4, we apply 
Crawley-Boevey's result \cite{CB} to deduce the FFRT property of $R=R(\PP^1,D)$ when $\delta_\mathfrak{C}<0$. 
In Section 5, we study the case $\delta_\mathfrak{C}=0$, using a covering $\varphi$ mentioned above, to prove that 
$\mathfrak{C}$ does not have global FFRT property. 
In Section 6, we slightly generalize Sun's result \cite{Su} on the stability of Frobenius push-forwards to 
orbifold curves with $\delta_\mathfrak{C}>0$. 
In Section 7, we summarize the result obtained in the previous 
sections with the main theorem (Theorem \ref{MainThm}) and discuss the exceptional cases where $p$ 
divides denominators of the $\Q$-divisor $D$. 

\begin{acknowledgements} 
NH thanks Holger Brenner for many useful comments to the manuscript.
RO thanks Isamu Iwanari for answering many questions on stacks, and Masao Aoki for calling his attention to the book \cite{Ol} which was helpful in writing this paper.
He also thanks Shunsuke Takagi for informing him of an example in \cite[Remark 3.4. (2)]{TT}.
The authors are grateful to the referees for their careful reading of the paper.
NH is partially supported by Grant-in-Aid for Scientific Research 16K05092, JSPS.
RO is partially supported by a Waseda University Grant for Special Research Projects (Project number: 2017S-077).
\end{acknowledgements}


\section{Preliminaries}
In this paper, we fix an algebraically closed field $k$.
The definition of the FFRT property requires the ring $R$ under consideration to be complete local or 
graded. In this paper we focus on the graded case, as follows. 
Let $R=\bigoplus_{m \ge 0} R_{m}$ be a Noetherian  normal $\mathbb N$-graded ring  over an algebraically 
closed field $R_0=k$ with $\dim R \ge 2$. 
We denote by $X$ the normal projective variety $X=\Proj R$.


\subsection{Pinkham-Demazure construction of a normal graded ring}
By \cite{D}, \cite{P} the graded ring $R$ is described as follows: There exists an ample $\Q$-Cartier divisor 
$D$ on $X$ such that 
$$R \cong R(X,D) = \bigoplus_{m \ge 0} H^0(X,\mathcal{O}_X(\lfloor mD \rfloor))t^m,$$
where $t$ is a homogeneous element of degree $1$. We write the $\Q$-divisor $D$ as 
$$D=\sum_{i=1}^n \frac{s_i}{r_i} D_i,$$ 
where $D_1,\dots,D_n$ are distinct prime divisors on $X$, and $r_i>0$ and $s_i$ are coprime integers. 

In the notation above, let
$$Y=\mathrm{Spec}_X\left(\bigoplus_{m\ge 0} \mathcal{O}_X(\lfloor mD\rfloor)t^m\right) \  
\mathop{\mathrm{and}}\nolimits\ 
   U=\mathrm{Spec}_X\left(\bigoplus_{m\in\Z} \mathcal{O}_X(\lfloor mD\rfloor)t^m\right).$$
Then $U$ is an open subset of $Y$ and we have the following commutative diagram. 
\begin{eqnarray}
\label{fund1}
\xymatrix{
& U \ar@{^(->}[d]  \ar[r]^-\cong & Z\setminus V(R_+) \ar@{^(->}[d]\\
\text{Ex}( \varphi) \ar[dr]^{\cong} \ar@{^(->}[r] & Y \ar[d]^{\sigma} \ar[r]^-\varphi & Z=\Spec R \\
 & X & 
}
\end{eqnarray}
Here $\mathop{\mathrm{Ex}}\nolimits(\varphi) =Y\setminus U$ is endowed with reduced closed 
subscheme structure. Then $\mathop{\mathrm{Ex}}\nolimits(\varphi) $ is a section of the structure 
morphism $\sigma\colon Y \to X$ and also the exceptional divisor of the graded blowup 
$\varphi \colon Y \cong \Proj(\bigoplus_{m \ge 0} R_{\ge m}) \to \Spec R$, where 
$R_{\ge m}  = \bigoplus_{m' \ge m} R_{m'}$. 
Also, $\sigma \colon Y \to X$ has an $\mathbb{A}^1$-bundle structure over the locus $X' \subseteq X$ 
where $D|_{X'}$ is an integral Cartier divisor. Namely, $\sigma^{-1}(X')$ is the line bundle associated to 
the invertible sheaf $\mathcal{O}_{X'}(D|_{X'})$ on $X'$. 
On the other hand, if we denote by $F_i$ the reduced fiber of $\sigma$ over the prime divisor $D_i$, 
then $\sigma^*D_i=r_iF_i$ (see \cite{D}). 


\subsection{Finite $F$-representation type}
We assume that the characteristic of $k$ is $p>0$. Then any scheme $S$ over $k$ admits the Frobenius 
morphism $F \colon S \to S$ associated with the $p$-th power ring homomorphism $\mathcal{O}_S \to F_*\mathcal{O}_S$. 
By our assumption, the graded ring $R$ is $F$-finite, i.e., the Frobenius on $Z=\Spec R$ is a finite morphism. 
For each $e=0,1,2,\dots$, the $e$-times Frobenius push-forward $F^e_*R$ of the graded ring $R$ is 
identified with the ring $R^{1/p^e}$, which has a natural $\frac1{p^e}\Z$-grading. Hence we can consider 
$R^{1/p^e}$ as an object of the category of finitely generated $\Q$-graded $R$-modules. In this category, 
we define an equivalence $\sim$ of objects by a graded isomorphism which admits a degree shift: 
Namely, 
for $\Q$-graded modules $M,N$, we define $M \sim N$ if $N \cong M(\alpha)$ via a degree-preserving 
isomorphism for some $\alpha\in\Q$. Now by the Krull-Schmidt theorem, we have a unique decomposition 
$$R^{1/p^e} = M_1^{(e)} \oplus\cdots\oplus M_{m_e}^{(e)}$$
in the category of finitely generated $\Q$-graded $R$-modules for $e=0,1,2,\dots$, with $M_i^{(e)}$ 
indecomposable. 

\begin{definition}[\cite{SVdB}]
\label{ffrt}
We say that $R$ has \emph{finite $F$-representation type} $($FFRT$)$ if the set of equivalence 
classes $\{M_i^{(e)} \,|\, e=0,1,2,\dots; i=1,\dots,m_e\}/ \sim$ is finite. 
\end{definition}

For $q=p^e$ we want to know the decomposition of the $R$-module $R^{1/q}$. 
The graded ring structure 
of $R=R(X,D)$ allows us to decompose 
$R^{1/q} = \bigoplus_{l\ge0} H^0(X,F^e_*\mathcal{O}_X(\lfloor lD\rfloor))t^{l/q}$
as $R^{1/q} = \bigoplus_{i=0}^{q-1} (R^{1/q})_{i/q \mod \Z}$, where 
$$
(R^{1/q})_{i/q \mod\Z} 
            = \bigoplus_{0 \le l \equiv i \mod q} H^0(X,F^e_*\mathcal{O}_X(\lfloor l D\rfloor))t^{l /q} 
            \cong \bigoplus_{m\ge 0} H^0(X,F^e_*\mathcal{O}_X(\lfloor (qm+i)D\rfloor)) 
$$
is an $R$-summand of $R^{1/q}$ for $i=0,1,\dots,q-1$. If $D$ is an integral Cartier divisor, then 
$F^e_*\mathcal{O}_X(\lfloor (qm+i)D\rfloor) \cong \mathcal{O}_X(D)^{\otimes m}\otimes F^e_*\mathcal{O}_X(iD)$ by the projection 
formula. 
Thus in this case, the decomposition of the $R$-module $(R^{1/q})_{i/q \mod\Z}$ depends on 
the decomposition of the vector bundle $F^e_*\mathcal{O}_X(iD)$ on $X$. However, this observation fails when $D$ 
is not integral. To overcome this difficulty, we will introduce a root stack associated with the pair $(X,D)$, 
which allows us to treat $D$ as if it is an integral divisor. 


\subsection{Algebraic stacks}
\label{alg}
In the remaining part of this section (subsections 2.3--2.8), we roughly review 
the construction and basic notions of stacks under the philosophy that most 
of known results in the category of schemes naturally extend to stacks. The 
contents of this section may be known to specialists, but we will put them to 
keep the consistency of the paper. 
We recommend readers unfamiliar 
with stacks to skip this section, just keeping in mind the above philosophy, 
and refer back this section when it is necessary. 
Our exposition and notation here are based on [Ol], [V2]; see also references therein for more 
details.

To introduce algebraic stacks, we recall terminology of categories.
\emph{Groupoids} are categories whose morphisms are all isomorphisms.
The simplest groupoids are sets, that is, categories whose morphisms are only identities.

For a category $C$, a \emph{category over} $C$ is a pair $(D, \pi)$, where $D$ is a category, and  $\pi \colon D \to C$ is a functor.
For an object $u$ and a morphism $\phi$ in $D$, we write $u \mapsto U$ and $\phi \mapsto f$  when $\pi(u)=U$ and $\pi(\phi)=f$.
A morphism $\phi \colon u \to v$ in $D$ is called \emph{cartesian} if for any other object $w$ in $D$ and a morphism $\psi \colon w \to v$ in $D$ with a factorization 
$$
\pi(\psi) \colon \pi(w) \stackrel{h}{\to} \pi(u) \stackrel{\pi(\phi)}{\to} \pi(v),
$$ 
there exists a unique morphism $\lambda \colon w \to u$ such that $\psi = \phi \circ \lambda$ and $\pi(\lambda) = h$.
In a picture:
$$
\xymatrix{
w \ar@{|->}[d] \ar@{-->}[r] \ar@/^18pt/[rr]^{\psi} & u \ar@{|->}[d] \ar[r]^{\psi} & v \ar@{|->}[d]\\
\pi(w) \ar[r]^{h} & \pi(u) \ar[r]^{\pi(\psi)} & \pi(v).}
$$

For an object $U$ in $C$, we define a category $D(U)$ as follows.
Objects of $D(U)$ are objects $u$ in $D$ such that $\pi(u)=U$.
For $u, u'$ in $F(U)$, morphisms $u \to u'$ in $D(U)$ are morphisms $\phi \colon u \to u'$ in $D$ such that $\pi(\phi) = \id_{U}$. 

\begin{definition}\label{fiber}
\begin{enumerate}
\item A \emph{fibered category over} $C$ is a category $\pi \colon D \to C$ over $C$ such that for every $v$ in $D(V)$ and morphism $f \colon U \to V$ in $C$, there exists a cartesian morphism $\phi \colon u \to v$ such that $\pi(\phi)=f$ (so in particular $u \in D(U)$).
\item A \emph{category fibered in groupoids over} $C$ is a fibered category $\pi \colon D \to C$ such that for every objct $U$ in $C$, a category $D(U)$ is a groupoid.
\end{enumerate}
\end{definition}

In Definition \ref{fiber} (1), the object $u$ is called a \emph{pull-back} of $v$ by $f$.
This is unique up to a unique isomorphism.
We also call $\pi$ a \emph{structure morphism} of the fibered category.
A morphism of categories fibered in groupoids is defined to be a functor strictly compatible with structure morphisms. 
If $g, g' \colon D_{1} \to D_{2}$ are morphisms of fibered categories $\pi_{i} \colon D_{i} \to C$ over $C$ for $i=1,2$, then a \emph{base preserving natural transformation} $\alpha \colon g \to g'$ is a natural transformation such that for each object $u$ of $D_{1}(U)$, the corresponding isomorphism $\alpha(u) \colon g(u) \cong g'(u)$ is in $D_{2}(U)$, that is, $\pi_{2}(\alpha(u))=\id_{U}$. 
Fiber products of categories fibered in groupoids are defined in \cite[3.4.12]{Ol}.
When a category $C$ has a Grothendieck topology, a category fibered in groupoids over $C$ is called a \emph{stack} over $C$ if it satisfies certain descent condition with respect to the topology of $C$ (cf. \cite[4.6.1]{Ol}).

In the following, we consider stacks over the category $\Sch_{k}$ of schemes over an algebraically closed field $k$ with the \'etale topology. 
So in particular, we have structure morphisms $\pi \colon \mathfrak{X} \to \Sch_{k}$, and for any scheme $S$ over $k$ we have a groupoid $\mathfrak{X}(S)$.
Isomorphism classes of objects in $\mathfrak{X}(S)$ are called $S$-valued points of $\mathfrak{X}$.
Morphisms $\mathfrak{X} \to \mathfrak{Y}$ of stacks $\mathfrak{X}$ and $\mathfrak{Y}$ over $k$ are morphisms of categories fibered in groupoids.
In particular, a morphism induces functors $\mathfrak{X}(S) \to \mathfrak{Y}(S)$ for schemes $S$.
For two morphisms $\varphi, \varphi' \colon \mathfrak{X} \to \mathfrak{Y}$, we consider base preserving natural transformations $\alpha \colon \varphi \cong \varphi'$.
They are simply called isomorphisms between $\varphi$ and $\varphi'$.

Morphisms $\mathfrak{X} \to \mathfrak{Y}$ are called \emph{representable} if for any $k$-scheme $S$ the fiber product $\mathfrak{X} \times_{S} \mathfrak{Y}$ is an algebraic space (cf. \cite[Ch 5]{Ol} for algebraic spaces).
For representable morphisms, certain properties of morphisms of algebraic spaces (such as \emph{smooth}) are defined using pull-back by schemes (cf. \cite[8.2.9]{Ol}).

Stacks $\mathfrak{X}$ over $\Sch_k$ are called \emph{algbraic} if they satisfy the following conditions (cf. \cite[8.1.4]{Ol}):
\begin{enumerate}
\item the diagonal morphism $\Delta \colon \X \to \X \times_{k} \X$ is representable
\item there exists a smooth surjection $U \to \X$ from a scheme $U$. 
\end{enumerate}
The first condition implies that any morphism $S \to \X$ from a scheme $S$ is representable.
When we can take a smooth (resp. locally of finite type) scheme as $U$ in the second condition, we say that $\X$ is \emph{smooth} (resp. \emph{locally of finite type}).
Algebraic stacks $\mathfrak{X}$ are called \emph{Deligne-Mumford} if we have an \'etale surjection $U \to \mathfrak{X}$ from a scheme $U$. 

Any $k$-scheme $S$ gives a category $\Sch_{/S}$ fibered in groupoids over $\Sch_{k}$, where $\Sch_{/S}$ is the category of $S$-schemes, and a structure morphism $\Sch_{/S} \to \Sch_{k}$ is defined by forgetting $S$-scheme structures.
We identify $S$ with $\Sch_{/S}$, and view $S$ as a Deligne-Mumford algebraic stacks.
We also define a groupoid $\Hom_{\Sch_{k}}(S, \mathfrak{X})$, whose objects are morphisms $S \to \mathfrak{X}$ of algebraic stacks.
For such morphisms $\varphi_{1}, \varphi_{2} \colon S \to \mathfrak{X}$, an isomorphism $\alpha \colon \varphi_{1} \cong \varphi_{2}$ in $\Hom_{\Sch_{k}}(S, \mathfrak{X})$ is defined to be a base preserving natural transformation.

By $2$-Yoneda lemma \cite[3.6.2]{V2}, we have an equivalence $\Hom_{\Sch_{k}}(S, \mathfrak{X}) \to \mathfrak{X}(S)$ of groupoids sending $\varphi $ to $\varphi(\id_{S})$.
In the following, we identify these two groupoids via this equivalence.


\subsection{Sheaves on algebraic stacks}
\label{coh}
For an algebraic stack $\mathfrak{X}$ over $k$, we define the \emph{lisse-\'etale site} $\Lis(\mathfrak{X})$ \emph{of} $\X$ as follows.
An object of $\Lis(\mathfrak{X})$ is an $\mathfrak{X}$-scheme $(T,t)$, where $T$ is a $k$-scheme and $t \colon T \to \mathfrak{X}$ is a smooth morphism.
A morphism $(T',t') \to (T, t)$ in $\Lis(\mathfrak{X})$ is a pair $(f, f^{b})$, where $f \colon T' \to T$ is a morphism of schemes and $f^{b}$ is an isomorphism $t' \to t \circ f$ of functors.
A collection $\lbrace (f_{i}, f^{b}_{i}) \colon (T_{i}, t_{i}) \to (T, t) \rbrace$ is defined to be a \emph{covering} in $\Lis(\X)$ if $\lbrace f_{i} \colon T_{i} \to T \rbrace$ is an \'etale covering.

Sheaves on $\X$ are defined to be sheaves on the site $\Lis(\X)$.
For example, we define $\mathcal{O}_{\X}$ by assigning $\Gamma(T, \mathcal{O}_{T})$ to any object $(T, t)$ of $\Lis(\X)$, and the pull-back $\Gamma(T, \mathcal{O}_{T}) \to \Gamma(T', \mathcal{O}_{T'})$ to any morphism $(f, f^{b}) \colon (T' , t') \to (T, t)$ in $\Lis(\X)$. 
For any sheaf $\F$ of $\mathcal{O}_{\X}$-module on $\X$ and object $(T, t)$ of $\Lis(\X)$, we define a sheaf $\F_{(T,t)}$ of $\mathcal{O}_{T}$-module on $T$ by the restriction of $\F$ to the \'etale site of $T$.
For each morphism $(f, f^{b}) \colon (T',t') \to (T, t)$, we have a natural morphism $\rho_{(f, f^{b})} \colon f^{\ast} \F_{(T,t)} \to \F_{(T',t')}$.
We say that $\F$ is \emph{quasi-coherent} if $\F_{(T, t)}$ is quasi-coherent for any object $(T,t)$, and $\rho_{(f, f^{b})}$ is an isomorphism for any morphism $(f, f^{b}) \colon (T' , t') \to (T, t)$ in $\Lis(\X)$. 
We write by $\text{QCoh} \X$ the category of quasi-coherent sheaves on $\X$.
A quasi-coherent sheaf $\F$ on $\X$ is called \emph{locally free} if each $\F_{(T,t)}$ is locally free.
%
%

Let $X \to \X$ be a smooth surjection with $X$ an algebraic space.
In the following, we define a coskelton of $X \to \X$, that is, a functor $X_{\bullet}$ from the opposite category of the simplicial category $\Delta$ to the category of algebraic spaces.
Here objects of $\Delta$ are finite posets $[n]=\lbrace 0, 1, \ldots, n \rbrace$, and morphisms in $\Delta$ are order preserving maps.
To each $[n]$, we assign 
\begin{eqnarray}
\label{coskelton}
X_{n}=\overbrace{X \times_{\X} \cdots \times_{\X} X}^{n+1},
\end{eqnarray}
and to a morphism $\delta \colon [n] \to [m]$, we assign 
$$
X_{\bullet}(\delta) = \p_{\delta(0)} \times \cdots \times \p_{\delta(n)} \colon X_{m} \to X_{n},
$$
where $\p_{i}$ denotes the $i$-th projection $X_{m} \to X$.
For each $[n]$, let $e_{n} \colon X_{n} \to \X$ denote the augmentation morphism such that $e_{n} \circ X_{\bullet} (\delta)=e_{m}$ for every morphism $\delta \colon [n] \to [m]$ in $\Delta$.

For quasi-coherent sheaves $\F$ on $\X$, data $\lbrace \F_{(X_{n}, e_{n})}, \rho_{(X_{\bullet}(\delta), \id)}\rbrace$ give quasi-coherent $\mathcal{O}_{X_{\bullet}}$-modules, whose definition is in \cite[9.2.12]{Ol}. 
This gives an equivalence of categories 
\begin{eqnarray}
\label{simplicial}
r \colon \text{Qcoh} \X \to \text{Qcoh} X_{\bullet},
\end{eqnarray}
where $\text{Qcoh} X_{\bullet}$ is the category of quasi-coherent $\mathcal{O}_{X_{\bullet}}$-modules.
Push-forward and pull-back of quasi-coherent sheaves are defined in \cite[9.2.5, 9.3]{Ol}.

When $\X$ is Deligne-Mumford, we can define quasi-coherent sheaves on $\X$ using the \'etale site of $\X$ by \cite[Proposition 9.1.18]{Ol}. 
For example, we have a quasi-coherent sheaf $\Omega_{\X/k}$ of differentials on $\X$ as in \cite[(7.20) (ii)]{V}, and differential maps $d \colon \mathcal{O}_{\X} \to \Omega_{\X/k}$.
A Deligne-Mumford stack $\X$ is smooth if and only if $\Omega_{\X/k}$ is locally free.
In this case, we put $\omega_{\X} = \det \Omega_{\X/k}$ and call it the \emph{canonical sheaf} on $\X$.

Proper morphisms of algebraic stacks are defined as in \cite[10.1.5]{Ol}. 
When $\X$ is proper over $k$, a quasi-coherent $\mathcal{O}_{\X}$-module $\F$ is called \emph{coherent} if $\F_{(T,t)}$ is coherent for each object $(T, t)$ in $\Lis(\X)$.
We write by $\Coh\X$ the category of coherent sheaves on $\X$.

When $\X$ is smooth proper Deligne-Mumford stacks, by \cite[Theorem 2.22]{N2} we have a Serre duality isomorphism from $\Ext_{\mathfrak{X}}^{i}(\F, \F')$ to the dual of  $\Ext_{\mathfrak{X}}^{d-i}(\F', \F \otimes \omega_{\mathfrak{X}})$ for coherent sheaves $\F, \F'$ on $\mathfrak{X}$, where $d=\dim \mathfrak{X}$.


\subsection{Affine morphisms of algebraic stacks}
\label{aff}

A morphism $f \colon \mathfrak{Z} \to \X$ of algebraic stacks is \emph{affine} (resp. \emph{finite}), if for any morphism $S \to \X$ from a scheme $S$, the pull-back $\mathfrak{Z} \times_{\X} S \to S$ is affine (resp. finite).

By \cite[10.2.4]{Ol}, for a finite morphism $f \colon \mathfrak{Z} \to \X$, we have an equivalence from $\text{Coh}\mathfrak{Z}$ to the category of coherent $f_{\ast} \mathcal{O}_{\mathfrak{Z}}$-modules on $\X$ as in \cite[Ex. II, 5.17e]{Ha}.
Hence we can construct a functor $f^{!} \colon \text{Coh} \X \to \text{Coh}\mathfrak{Z}$ such that we have
$$
f_{\ast} \mathcal{H}om_{\mathfrak{Z}}(\F, f^{!} \mathcal{G}) \cong \mathcal{H}om_{\X}(f_{\ast} \F, \G)
$$
for $\F$ in $\text{Coh}\mathfrak{Z}$ and $\G$ in $\text{Coh}\X$ as in \cite[III, Ex. 6.10]{Ha}.

In particular, when $\X$ is a smooth Deligne-Mumford stack, we see that we have an adjunction isomorphism 
\begin{eqnarray}
\label{adjunction}
\mathcal{H}om_{\X}(f_{\ast} \F, \omega_{\X}) \cong \F^{\vee} \otimes f_{\ast} \omega_{\mathfrak{Z}} 
\end{eqnarray} 
for any locally free sheaf $\F$ on $\mathfrak{Z}$ by the similar arguments as in \cite[III, Ex. 7.2]{Ha}.


\subsection{Frobenius morphisms}
\label{frob}
When the characteristic $p$ of $k$ is positive, we define \emph{Frobenius morphisms} $F \colon \mathfrak{X}^{(1)} \to \mathfrak{X}$ of algegraic stacks $\mathfrak{X}$ over $k$ as follows.
Here $\mathfrak{X}^{(1)}$ is equal to $\mathfrak{X}$ as categories, but we define a structure morphism $\mathfrak{X}^{(1)} \to \Sch_{k}$ by composing structure morphisms $\mathfrak{X} \to \Sch_{k}$ with the usual Frobenius morphism $F \colon \Sch_{k} \to \Sch_{k}$. 
It is easily shown that $\mathfrak{X}^{(1)}$ are algebraic stacks over $k$.
We can regard objects of $\mathfrak{X}^{(1)}$ over $\sigma^{(1)} \colon S \to \Spec k$ as pairs $(\sigma, \xi)$ of $\sigma \colon S \to \Spec k$ such that $F \circ \sigma = \sigma^{(1)}$, and objects $\xi$ of $\mathfrak{X}$ over $\sigma \colon S \to \Spec k$.
 
In this subsection, for a scheme $\sigma \colon S \to \Spec k$ over $k$, we often write $F \circ \sigma$ by $\sigma^{(1)}$, and $(S, \sigma^{(1)})$ by $S^{(1)}$.  
We note that since $\X$ and $\X^{(1)}$ are equal as categories (or algebraic stacks over $\Z$), some notion such as the category of quasi-coherent sheaves and the Chow ring of them are naturally identified.

For any object $u=( \sigma, \xi)$ in $\mathfrak{X}^{(1)}(S^{(1)})$, we define an object $F(u)$ in $\mathfrak{X}(S^{(1)})$ by the pull-back of $\xi$ by the Frobenius morphism $F \colon S^{(1)} \to S$ so that we have a cartesian morphism $F(u) \to \xi$ in $\mathfrak{X}$.
For a morphism $\phi \colon u \to v$ in $\mathfrak{X}^{(1)}$, we define a morphism $F(\phi) \colon F(u) \to F(v)$ in $\mathfrak{X}$ by the universal property of the cartesian morphism $F(v) \to v$.

The definition of $F \colon \mathfrak{X}^{(1)} \to \mathfrak{X}$ depends on choices of pull-backs $F^{\ast} \xi$ for each $u=(\sigma, \xi) \in \mathfrak{X}^{(1)}$.
But if we define $F_{i} \colon \mathfrak{X}^{(1)} \to \mathfrak{X}$ for $i=1,2$ by different choices, then we have a unique isomorphism $F_{1} \cong F_{2}$ compatible with cartesian morphisms $F_{i}(u) \to \xi$.
Similarly, for a morphism $\psi \colon \mathfrak{Y} \to \mathfrak{X}$ of algebraic stacks, we also have a commutative diagram
\begin{eqnarray}
\label{stackdiagram}
\xymatrix{
\mathfrak{Y}^{(1)} \ar[r]^{F} \ar[d]^{\psi} & \mathfrak{Y} \ar[d]^{\psi} \ar@{<=}[ld]\\
\mathfrak{X}^{(1)} \ar[r]^{F} & \mathfrak{X},
}
\end{eqnarray}
where $\psi \colon \mathfrak{Y}^{(1)} \to \mathfrak{X}^{(1)}$ is the same functor as $\psi \colon \mathfrak{Y} \to \mathfrak{X}$.

\begin{example}
\label{quot}
For a group scheme $G$ acting on a scheme $U$ over $k$, we write by $[U/G]$ the quotient stack of $U$ by $G$ as in \cite[Example 8.1.12]{Ol}. 
We describe the Frobenius morphism of a quotient stack $\X=[U / G]$.
Since $k$ is algebraically closed, the Frobenius morphism $F \colon \Spec k \to \Spec k$ is an isomorphism, and we have a natural isomorphism $\Psi \colon [U^{(1)} / G^{(1)}] \cong \X^{(1)} $.

We write by $\Phi \colon [U^{(1)}/G^{(1)}] \to \X=[U/G]$ the induced morphism by $F \colon U^{(1)} \to U$ which is equivariant through $F \colon G^{(1)} \to G$.
Then we have an isomorphism $\eta \colon \Phi \cong F \circ \Psi$, 
that is, we have the following commutative diagram:
$$
\xymatrix{
[U^{(1)}/G^{(1)}] \ar[dd]_{\Phi} \ar[r]^-{\Psi} & \X^{(1)} \ar[dd]^{F} \\
\ar@{=>}[r]^{\eta}& \\
[U / G] \ar@{=}[r] & \X 
}
$$
Hence we have $\X^{(1)} \times_{\X} U \cong [U^{(1)} \times G / G^{(1)}]$.

When $F \colon G^{(1)} \to G$ is an isomorphism, then we have $\mathfrak{X}^{(1)} \times_{\mathfrak{X}} U \cong U^{(1)}$, and the diagram \eqref{stackdiagram} for $\mathfrak{Y}=U$ is cartesian.
This is the case when $G=\mu_{m} = \Spec k[\zeta]/\langle \zeta^{m}-1\rangle$ for a positive integer $m$ co-prime to $p$.  
An orbifold curve, introduced in the next section, with the weights 
co-prime to $p$ is locally isomorphic to such a quotient stack. 

On the other hand, when $G=\mu_{p} = \Spec k[\zeta]/\langle \zeta^{p}-1\rangle$, the Frobenius morphism $G^{(1)} \to G$ is decomposed through the unit map $\Spec k \to  G$. 
Hence we have $\X^{(1)} \times_{\X} U \cong \X^{(1)} \times G$.
In particular, the Frobenius morphism $\X^{(1)} \to \X$ is not affine in this case.
\end{example}

When $\X$ is Deline-Mumford, we have an \'etale surjection $X \to \X$ and coskeltons $X_{\bullet}$ and $X_{\bullet}^{(1)}$ as in \eqref{coskelton}.
We further assume that $\X$ is locally a quotient stack of a smooth scheme by $\mu_{m}$ for a positive integer $m$ co-prime to $p$.
Then $F \colon \X^{(1)} \to \X$ is representable, and the diagram \eqref{stackdiagram} is cartesian for $\mathfrak{Y} = X$ by Example \ref{quot}. 
Hence we have a cartesian diagram
$$
\xymatrix{
X_{m}^{(1)} \ar[r]^{F} \ar[d]_{X_{\bullet}(\delta)} & X_{m} \ar[d]^{X_{\bullet}(\delta)} \\
X_{n}^{(1)} \ar[r]^{F} & X_{n},
}
$$
for any morphism $\delta \colon [n] \to [m] $ in $\Delta$.
We can compute the push-forward $F_{\ast}$ and the pull-back $F^{\ast}$ of the Frobenius morphism $F \colon \X^{(1)} \to \X$ by $F_{\ast} \colon \text{Qcoh} X_{n}^{(1)} \to \text{Qcoh} X_{n}$ and $F^{\ast} \colon \text{Qcoh} X_{n} \to \text{Qcoh} X_{n}^{(1)}$ via the equivalence \eqref{simplicial}.
We can also make various constructions of coherent sheaves on $\X$ and describe differential maps  in terms of $\mathcal{O}_{X_{\bullet}}$-modules via the equivalence \eqref{simplicial}.
Furthermore 
we can see that $F \colon \X^{(1)} \to \X$ is flat and finite since the pull-back $F \colon X^{(1)} \to X$ by an \'etale cover is flat and finite.
Thus the usual arguments for coherent sheaves concerning Frobenius morphisms also hold for such a Deligne-Mumford stack $\X$.

The similar argument holds as above for the iterated Frobenius morphism $F^{e} \colon \X^{(e)} \to \X$. 
Here  $e$ is a positive integer, and $\mathfrak{X}^{(e)}$ denotes a $k$-stack with the structure morphism equal to the composition of $\mathfrak{X} \to \Sch_{k}$ with $F^{e} \colon \Sch_{k} \to \Sch_{k}$. 
In the remaining part of the paper, we will simply write $\X=\X^{(e)}$ since any confusion is not supposed to occur.
We also use notation $F_{\X} \colon \X \to \X$ for the Frobenius morphism $F$, when we want to emphasize $\X$.

\subsection{Root stacks}
Here we briefly review the notion of root stacks (cf. \cite[10.3]{Ol}).
\begin{definition}
Let $\X$ be an algebraic stack. 
A \emph{generalized effective Cartier divisor on} $\X$ is a pair $D=(L, \rho)$, where $L$ is a line bundle on $\X$, and $\rho \colon L \to \mathcal{O}_{\X}$ is a morphism of $\mathcal{O}_{\X}$-modules.  
If $D=(L,\rho)$ and $D'=(L', \rho')$ are two generalized effective Cartier divisors on $\X$, then an isomorphism between them is an isomorphism $\sigma \colon L \to L'$ of $\mathcal{O}_{\X}$-modules such that the following diagram commutes:
$$
\xymatrix{
L \ar[rr]^{\sigma} \ar[dr]_{\rho} && L' \ar[dl]^{\rho'} \\
& \mathcal{O}_{\X} &
}
$$ 
\end{definition}

In the above definition, we often write $L$ by $\mathcal{O}_{\X}(-D)$, and the dual by $\mathcal{O}_{\X}(D)$.
We define addition of $D=(L,\rho)$ and $D'=(L', \rho')$ by $D+D'=(L \otimes L', \rho \otimes \rho')$.
For a non-negative integer $r$, we put $r D= (L^{\otimes r}, \rho^{\otimes r})$.
 
Let $X$ be a $k$-scheme, and $D_{1}, \ldots, D_{n}$ generalized effective Cartier divisors on $X$.
We consider the associated section $\xi_{i} \colon \mathcal{O}_{X}(-D_{i}) \to \mathcal{O}_{X}$ and the induced morphism 
$\mbi \xi \colon X \to [\mathbb{A}^1/k^{\ast}]^{n}$, where $\mathbb{A}^1=\mathbb{A}^1_k$ is the affine line on which $k^*$ acts by multiplication. 

For $\mbi r=(r_1,\ldots, r_n) \in \Z_{\ge 0}^n$, an \emph{$\mbi r$-th root stack} $\pi \colon \mathfrak{X} \to X$ 
of $X$ along $(D_{1}, \ldots, D_{n})$ is defined as follows.
For a $k$-scheme $S$, objects of the groupoid $\mathfrak{X}(S)$ are data 
$$
\xi = (f \colon S \to X, E_{1}, \ldots, E_{n}, \alpha_{1}, \ldots, \alpha_{n}), 
$$
where $E_{i}$ are generalized effective Cartier divisors on $S$, and $\alpha_{i} \colon r_{i} E_{i} \cong f^{\ast} D_{i}$ are isomorphisms.

For another  object $\xi' = (f' \colon S \to X, E'_{1}, \ldots, E'_{n}, \alpha'_{1}, \ldots, \alpha'_{n}) \in \mathfrak{X}(S')$, a morphism $\eta \colon \xi \to \xi' $ in $\mathfrak{X}$ is defined to be a pair $(\varphi, (\eta_{i})_{i=1}^{n})$, where $\varphi \colon S \to S'$ is a morphism of $X$-schemes, and $\eta_{i}$ are isomorphisms $E_{i} \cong \varphi^{\ast} E'_{i}$ such that the following diagrams commute:
$$
\xymatrix{
\ar[d]^-{\eta_{i}^{\otimes r_{i}}} r_{i} E_{i} \ar[r]^-{\alpha_{i}} & f^{\ast} D_{i} \ar[d]^{\cong}\\
r_{i} \varphi^{\ast} E_{i}' \ar[r]^-{\varphi^{\ast} \alpha'_{i}} & \varphi^{\ast} f'^{\ast} D_{i}
}
$$
We define a structure morphism $\pi \colon \X \to X$ by $\pi(\xi) = f$.
Clearly we have universal generalized effective Cartier divisors $\E_{1}, \ldots, \E_{n}$ on $\mathfrak{X}$ such that $r_{i} \mathcal{E}_{i} = \pi^{\ast} D_{i}$.

As in the proof of \cite[Theorem 10.3.10]{Ol}, we can regard $\X$ as the pull-back of $\mbi r$ by $\mbi \xi$
 \begin{eqnarray}\label{root}
\xymatrix{
\ar[d]_-\pi \mathfrak{X} \ar[r]^-{\mbi \xi_{\mathfrak{X}}} & [\mathbb{A}/ k^{\ast}]^{n} \ar[d]^-{\mbi r}\\
X \ar[r]^-{\mbi \xi} & [\mathbb{A}/ k^{\ast}]^{n},
}
\end{eqnarray}
\noindent 
where $\mbi r \colon [\mathbb{A}^1/ k^{\ast}]^{n} \to [\mathbb{A}^1/ k^{\ast}]^{n}$ is induced by $r_{i}$-th 
power $\mathbb{A}^1 \to \mathbb{A}^1$; $z \mapsto z^{r_i}$ for $i=1,\dots,n$. 
We also write 
$\mathfrak{X}=X[\sqrt[r_{1}]{D_{1}}, \ldots, \sqrt[r_{n}]{D_{n}}]$, and call $r_{1}, \ldots, r_{n}$ the \emph{weights}.
We can also regard $\mathcal{E}_i$ as the generalized effective Cartier divisor on $\mathfrak{X}$ defined by ${\mbi\xi}_\mathfrak{X}^*z_i$, 
where ${\mbi\xi}_\mathfrak{X}$ is in the diagram (\ref{root}), and $z_i$ is the coordinate of the $i$-th 
component in $[\mathbb{A}^1/k^*]^n$. 

Locally, we take an affine open subset $W=\Spec A$ of $X$ 
such that $D_i|_W=\{f_i=0\}$ for $f_i\in A$ and $i=1,\dots,n$. 
Since \eqref{root} is a Cartesian diagram, $\pi^{-1} W$ coincides with the root stack of $\Spec A$. 
Hence it is isomorphic to 
$[\Spec B/\mu_{r_1}\times\cdots\times\mu_{r_n}]$, where
\begin{eqnarray}
\label{local}
B=A[w_1,\dots,w_n]/\langle w_1^{r_1}-f_1,\dots,w_n^{r_n}-f_n\rangle,
\end{eqnarray}
and $\mu_{r_1}\times\cdots\times\mu_{r_n}$ acts on $\Spec B$ by 
$(w_1,\dots,w_n) \mapsto (\eta_1w_1,\dots,\eta_nw_n)$ for $(\eta_1,\dots,\eta_n) \in \mu_{r_1}\times\cdots\times\mu_{r_n}$ (cf. \cite[Theorem 10.3.10 (ii)]{Ol}). 

For later use, we summarize a few fundamental properties of root stacks $\pi \colon \mathfrak{X} \to X$ in 
the following:

\begin{lemma}\label{cadman}
Under the notation as above we have the following.
\begin{enumerate}
\item[$(1)$] $\pi \colon \mathfrak{X} \to X$ is an isomorphism away from $\E_{i}$ and $D_{i}$ $(i=1,\dots,n)$. 
\item[$(2)$] $\E_i$ is an generalized effective Cartier divisor on $\mathfrak{X}$ with $\pi^{\ast}D_i=r_i \E_i$ for $i=1, \ldots, n$.
\item[$(3)$] If $\pi' \colon \mathfrak{X}' \to X$ is a morphism such that there exist generalized effective Cartier divisors $\E'_i$ on $\mathfrak{X}'$ with 
$r_i \E'_i=(\pi')^*D_i$ for $i=1,\dots,n$, then there is a morphism $\varphi \colon \mathfrak{X}' \to \mathfrak{X}$ such that 
$\pi' \cong \pi\circ\varphi$ and $\E_{i}' = \varphi^{\ast} \E_{i}$.
This morphism $\varphi$ is unique up to isomorphisms.
\item[$(4)$] For a generalaized effective Cartier divisor $\sum_{i=1}^{n} l_{i} \E_{i}$ on $\mathfrak{X}$, one has 
$$\pi_*\mathcal{O}_\mathfrak{X}(\sum_{i=1}^{n} l_{i} \E_{i})=\mathcal{O}_{X}(\sum_{i=1}^{n}\lfloor \frac{l_{i}}{r_{i}} \rfloor D_{i}) \ \mathop{\mathrm{and}}\nolimits\  R^i\pi_*\mathcal{O}_\mathfrak{X}(\sum_{i=1}^{n} l_{i} \E_{i})=0 \mathop{\mathrm{ for }}\nolimits
 i>0.$$ 
\item[$(5)$] If the characteristic $p$ of the field $k$ does not divide any $r_{i}$, then $\mathfrak{X}$ is a Deligne-Mumford stack.
\end{enumerate}
\end{lemma}

\begin{proof} 
(1), (2) and (3) follow from the definition of the root stack $\mathfrak{X}$ and Cartier divisors $\E_i$ on $\mathfrak{X}$. 
As for (4), by the local description as above, coherent sheaves are considered as $\prod_{i=1}^{n}\mu_{r_i}$-equivariant coherent sheaves on $\Spec B$, where $B$ is in \eqref{local}, and the push-forward $\pi_{\ast}$ 
corresponds to taking $\prod_{i=1}^{n}\mu_{r_i}$-invariant parts. 
This is exact, hence $R^i\pi_{\ast}\F = 0$ for any coherent sheaf $\F$ on $\mathfrak{X}$ and $i>0$.

We write $l_i = r_i k_i + m_i$ for $0 \le m_i < r_i$, that is, $k_i = \lfloor \frac{l_i}{r_i} \rfloor$.
For a line bundle $\mathcal{O}_\mathfrak{X}(E)|_{\pi^{-1}W}$, the corresponding $B$-module is 
$$
B w_{1}^{-l_1} \cdots w_{n}^{-l_n} =  \bigoplus_{j_{1}=1}^{r_{1}-1} \cdots \bigoplus_{j_{n}=0}^{r_{n}-1} A\frac{w_{1}^{j_{1}} \cdots w_{n}^{j_{n}} }{f_{1}^{k_1}w_{1}^{n_1} \cdots f_{n}^{k_n}w_{n}^{m_n}}.
$$
Hence, as desired, the push-forward $\pi_*\mathcal{O}_\mathfrak{X}(\pi^*D)|_W$ is $\mathcal{O}_W(\sum_{i=1}^{n} \lfloor\frac{l_i}{r_i}\rfloor D_i)$ which correspond to the degree $0$ part $A\frac{1}{f_{1}^{k_{1}} \cdots f_{n}^{k_{n}}}$. 

(5) follows from \cite[Theorem 10.3.10 (iv)]{Ol}.
\end{proof}

\begin{example}
\label{sand}
We assume that the characteristic of $k$ is $p>0$, and consider a root stack $\pi \colon \X \to X$ whose weights are all equal to $p^{e}$.
Then $F^{e}=(F_{X})^{e} \colon X \to X$ induces a morphism $\varphi \colon X \to \X$ by Lemma \ref{cadman} (3) such that $(F_{X})^{e}=\pi \circ \varphi$.
Furthermore we also have $(F_{\X})^{e} \cong \varphi \circ \pi$, since for 
$$
\xi = (f \colon S \to X, E_{1}, \ldots, E_{n}, \alpha_{1}, \ldots, \alpha_{n}) \in \X(S), 
$$
we have $(F_{\X})^{e}(\xi)=(F^{e})^{\ast}\xi$ by definition, and $(F^{e})^{\ast} E_{i} \cong p^{e} E_{i} \stackrel{\alpha_{i}}{\cong} f^{\ast} D_{i}$ induces an isomorphism $(F_{\X})^{e}(\xi) \cong \varphi \circ \pi(\xi)$.
In this case, we say that $\X$ is a ``Frobenius sandwich'' of $X$; see Proposition \ref{propsand}.
\end{example}

\subsection{Sheaves of differentials on root stacks}
We continue the arguments and same notation as in the preceding subsection, but here we assume that the characteristic $p$ of $k$ does not divide any $r_{i}$.
Then by Lemma \ref{cadman} (5), the root stacks $\X$ are Delign-Mumford, and we have sheaves $\Omega_{\X/k}$ of differentials on $\X$. 

To describe $\Omega_{\X/k}$, we use another equivalent description of $\mathfrak{X}$ as follows. 
Let $L_{1}, \ldots, L_{n}$ be the total spaces of line 
bundles $\mathcal{O}_{X}(-D_{1}), \ldots, \mathcal{O}_{X}(-D_{n})$. We consider the complement $L_i^{\times}$ of the zero 
section. 
We consider $(k^{\ast})^{n}$-action on $U=L_1^{\times}\times_X\cdots\times_XL_n^{\times}\times\mathbb{A}^n$ by  $((t_i^{r_i}s_i),(t_ix_i))$ 
for $(t_{i}) \in (k^{\ast})^{n}$, and $(k^{\ast})^{n}$-equivariant vector bundle $E = U \times \prod_{i=1}^{n} k t_{i}^{r_{i}}$ over $U$.
We define an equivariant section $s \colon U \to E$ defined by 
$$
((s_{i}), (x_{i})) \mapsto (\xi_{i}(s_{i} ) - x_{i}^{r_{i}}) 
$$  
and a closed subscheme $V=s^{-1}(0)$ of $U$, where $\xi_{i} \colon \mathcal{O}_{X}(-D_{i}) \to \mathcal{O}_{X}$ is the section 
associated to the generalized effective Cartier divisor $D_{i}$ on $X$.
We put $\mathfrak{X}'=[V/(k^{\ast})^n]$, and consider the natural projection $\pi' \colon \mathfrak{X}' \to X$.

We consider Cartier divisors $\E_{i}'=\lbrace x_{i}=0 \rbrace$ on $\mathfrak{X}'$.
If we write $\mathcal{L}_{i}=\mathcal{O}_{\mathfrak{X}'}(\E_{i}')$, then we have an isomorphism $\mathcal{L}_{i}^{\otimes r_{i}} \cong (\pi')^{\ast} \mathcal{O}_{X}(D_{i})$ from the construction of $\mathfrak{X}'$.
This gives an homomorphism $\mathfrak{X}' \to \mathfrak{X}$ by Lemma \ref{cadman} (3).
By the above local description of $\mathfrak{X}$, this gives an isomorphism $\mathfrak{X}' \cong \mathfrak{X}$.
Hence we can identify coherent sheaves on $\X$ with $(k^{\ast})^{n}$-equivariant coherent sheaves on $V$.

We consider the following complex of $(k^{\ast})^{n}$-equivariant vector bundles on $V$:
\begin{eqnarray}
\label{cpx}
0 \to E^{\vee}|_{V} \to \Omega_{U/k}|_{V} \to \mathcal{O}_{V}^{ \oplus n} \to 0,
\end{eqnarray}
where we identify the fiber of the trivial bundle $\mathcal{O}_{V}^{\oplus n}$ as the cotangent space of $(k^{\ast})^{n}$ at the unit. 
For any \'etale morphism $S \to \X'$, we have a cartesian diagram
$$
\xymatrix{ P \ar[r] \ar[d]& V \ar[d] \\
S \ar[r] & \X',}
$$ 
where the vertical arrow $P \to S$ is a principal $(k^{\ast})^{n}$-bundle over $S$, and the horizontal arrow $P \to V$ is a $(k^{\ast})^{n}$-equivariant morphism.
We see that the pull-back of $\Omega_{S/k}$ to $P$ is naturally isomorphic to the middle cohomology of the pull-back of \eqref{cpx} as $(k^{\ast})^{n}$-equivariant coherent sheaves.
Hence the middle cohomology of the complex \eqref{cpx} is isomorphic to $\Omega_{\X/k}$.
Since $E^{\vee}|_V=\bigoplus_{i=1}^n \mathcal{O}_V t_i^{-r_i}$ and 
$\Omega_{U/k}=\bigoplus_{i=1}^n (\mathcal{O}_U\oplus \mathcal{O}_Ut_i^{-1}) \oplus \Omega_{X/k}|_U$, we have 
\begin{eqnarray}
\label{can}
\omega_{\mathfrak{X}} = \pi^{\ast} \det \Omega_{X/k} \otimes \mathcal{O}_{\mathfrak{X}}(\sum_{i=1}^{n} (r_{i} -1) \E_{i}).
\end{eqnarray}

When $D=\bigcup_{i=1}^{n} D_{i}$ is a divisor with normal crossings and $p$ does not divide any $r_{i}$ for $i=1, \ldots, n$, we also have another description of root stacks.
Under this assumption, we consider a Delgne-Mumford stack $\tilde{\pi} \colon \widetilde{\mathfrak{X}} \to X$ constructed in \cite[Section 4]{MO} together with a normal crossings divisor $\widetilde{\E}=\bigcup_{i=1}^{n} \widetilde{\E}_{i}$ such that $\tilde{\pi}^{\ast} \mathcal{O}_{X} (D_{i}) \cong \mathcal{O}_{\widetilde{\mathfrak{X}}}(r_{i} \widetilde{\E}_{i})$.
Hence by Lemma \ref{cadman} (3), we have a morphism $\widetilde{\mathfrak{X}} \to \mathfrak{X}$.
From the local description \cite[Lemma 4.3]{MO}, we see that this is an isomorphism.


\section{FFRT property of $R(C,D)$ via orbifold curves}


\subsection{Orbifold curves}
By an ``orbifold curve," we mean a one-dimensional smooth separated Deligne-Mumford stack $\mathfrak C$ 
whose coarse moduli map $\pi \colon \mathfrak{C} \to C$ to a smooth projective curve $C$ is generically isomorphism. 
As in \cite[1.3.6]{B}, an orbifold curve is a root stack over 
$C$. 

We fix the notation to be used throughout this section. 
Given integers $r_1,\dots,r_n \ge 2$ and closed points $P_1,\dots,P_n$ on a smooth projective curve 
$C$, let $\mathfrak{C}=C[\sqrt[r_1]{P_1},\dots,\sqrt[r_n]{P_n}\,]$ be the root stack of weight $(r_1,\dots,r_n)$ and let $\pi \colon \mathfrak{C} \to C$ be the coarse moduli map. \textrm{F}or $i=1,\dots,n$, we denote by 
$Q_i$ the stacky point over $P_i$, that is, the integral Cartier divisor on $\mathfrak{C}$ with $\pi^*P_i=r_iQ_i$. 
Since $\dim \mathfrak{C}=1$, we have $\omega_{\mathfrak{C}}=\Omega_{\mathfrak{C}}$ when $p$ does not divide any weight $r_{i}$, and the next lemma follows from \eqref{can}. 
We also use the notation $K_{\mathfrak{C}}$ for the canonical divisor. 

\begin{lemma}
\label{dualizing}
We assume that $p$ does not divide any weight $r_{i}$. 
Then we have 
$$
\omega_\mathfrak{C}=\Omega_{\mathfrak{C}} \cong \pi^* \omega_C \otimes \mathcal{O}_\mathfrak{C}(\sum_{i=1}^n (r_i-1)Q_i). 
$$
\end{lemma}

We consider the Chow ring $A(\mathfrak{C})$ with $\mathbb{Q}$-coefficients, and the map $\deg \colon A(\mathfrak{C}) \to A(\Spec k) \cong \Q$ induced by the push-forward by the structure morphism $\mathfrak{C} \to \Spec k$ (cf. \cite[Definition 1.15]{V}).
We have $\deg \omega_{\mathfrak{C}}=n + 2g - 2 - \sum_{i=1}^{n} \frac{1}{r_{i}}$ by Lemma \ref{dualizing}, where $g$ is the genus of $C$.

Now let $R=R(C,D)$ for an ample $\Q$-Cartier divisor $D=\sum_{i=1}^n (s_i/r_i)P_i$ on $C$ as in subsection 
2.1. Then Lemma \ref{cadman} (4) allows us to think of $R=R(C,D)$ as the section ring associated with an \emph{integral Cartier divisor} $\pi^*D$ or equivalently a line bundle $L=\mathcal{O}_\mathfrak{C}(\pi^*D)$ on $\mathfrak{C}$: 
$$R = R(C,D) \cong R(\mathfrak{C},L) = \bigoplus_{m \ge 0} H^0(\mathfrak{C},L^m)t^m.$$

We will extend the fundamental diagram (\ref{fund1}) to the stacky situation. 
$$\widetilde{Y}=\mathrm{Spec}_\mathfrak{C}\left(\bigoplus_{m \ge 0} L^m t^m\right) \  
\mathop{\mathrm{and}}\nolimits\ 
   \widetilde{U}=\mathrm{Spec}_\mathfrak{C}\left(\bigoplus_{m \in\Z} L^m t^m \right).$$
Then $\widetilde{Y}$ is an $\mathbb{A}^1$-bundle over $\mathfrak{C}$ and we have the following 
extended fundamental diagram. 
$$
\xymatrix{
\widetilde{U} \ar@{^(->}[d] \ar[r] & U \ar@{^(->}[d]  \ar[r]^-\cong & Z\setminus V(R_+) \ar@{^(->}[d]\\
\widetilde{Y} \ar[d]_{\widetilde{\sigma}} \ar[r]^{\psi} & Y \ar[d]^{\sigma} \ar[r]^-\varphi & Z=\Spec R \\
\mathfrak{C} \ar[r]^{\pi} & C & 
}
$$
Here $\psi \colon \widetilde{Y} \to Y$ and $\widetilde{U} \to U$ are induced by isomorphisms 
$\pi_{\ast} L^{m} \cong \mathcal{O}_{C}(\lfloor m D \rfloor )$ for $m \in \Z$ (cf.\ \cite[Theorem 10.2.4]{Ol}).

In the local description \eqref{local}, we see that $(\pi\circ\widetilde{\sigma})^{-1} W$ is a quotient stack 
$[V/\prod_{i=1}^{n} \mu_{r_i}]$ of a $\prod_{i=1}^{n} \mu_{r_i}$-equivariant line bundle $V$ over $\Spec B$, and 
$$
\sigma^{-1}(W)=\Spec \Gamma(V, \mathcal{O}_V)^{\prod_{i=1}^{n} \mu_{r_i}}
$$ 
by Lemma \ref{cadman} (4).  
Furthermore $\psi |_{(\pi\circ\widetilde{\sigma})^{-1}W}$ 
is given by the inclusion $\Gamma(V, \mathcal{O}_V)^{\prod_{i=1}^{n} \mu_{r_i}} \to \Gamma(V, \mathcal{O}_V)$.
Hence $\psi$ is a coarse moduli map (cf.\ \cite[Chapter 6]{Ol}). 

\begin{lemma}
\label{grading}
In the situation as above we have the following. 
\begin{enumerate}
\item[$(1)$] $\psi$ induces an isomorphism $\widetilde{U} \cong U$ preserving the $\Z$-grading. 
\item[$(2)$] We have an isomorphism
$$
\mathfrak{C} \cong [\widetilde{U}/k^{\ast}],
$$
where $k^{\ast}=\Spec k[t, t^{-1}]$ is the multiplicative group, and $k^{\ast}$-action on 
$\widetilde{U}$ is induced by the fiberwise multiplication on the line bundle $\widetilde{Y}$ 
over $\mathfrak{C}$.
\end{enumerate}
Consequently, we have an equivalence between the category of vector bundles on $\mathfrak{C}$ and 
the category of reflexive $\Z$-graded $R$-modules given by 
$$\E \mapsto \Gamma_*(\E)=\bigoplus_{m \in\Z} H^0(\mathfrak{C},\E\otimes L^m)t^m.$$ 
\end{lemma}

\begin{proof}
(1) Since $r_i$ and $s_i$ are co-prime to each other, the automorphism functors of all closed points of $\widetilde{U}$ are trivial. 
Hence by \cite[Theorem 2.2.5]{C}, $\widetilde{U}$ is an algebraic space, and the coarse 
moduli map $\psi|_{\widetilde{U}} \colon \widetilde{U} \to U$ is an isomorphism.
(2) is obvious.

It follows from (2) that there is a one-to-one correspondence between vector bundles on $\mathfrak{C}$ 
and $\Z$-graded vector bundles on $\widetilde{U}$ given by $\E \mapsto \bigoplus_{m \in\Z}\E\otimes L^m$. 
On the other hand, since $\widetilde{U} \cong Z\setminus V(R_+)$ by (1) and codim$(V(R_+),Z)=2$, 
we have a one-to-one correspondence between $\Z$-graded vector bundles on $\widetilde{U}$ 
and reflexive graded $R$-modules, from which the required correspondence follows. 
\end{proof}

\begin{definition}
For a line bundle $L$ on $\mathfrak{C}$, we define an equivalence $\sim_L$ in $\Coh(\mathfrak{C})$ as follows. 
For $\E,\,\F \in$ $\Coh(\mathfrak{C})$, we write $\E \sim_L \F$ if $\E \cong \F\otimes L^m$ 
for some $m \in \Z$.
\end{definition}

\begin{corollary}\label{1-1cor}
For a vector bundle $\E$ on $\mathfrak{C}$, 
$$\E \mapsto \Gamma_*(\E)=\bigoplus_{m \in\Z} H^0(\mathfrak{C},\E\otimes L^m)t^m$$ 
gives a one-to-one correspondence between the set of equivalence classes of vector bundles on $\mathfrak{C}$ with respect to $\sim_L$ and the set of equivalence classes of reflexive $\Z$-graded 
$R$-modules with respect to the equivalence $\sim$ admitting degree shift as in subsection 2.2.
\end{corollary}


\subsection{FFRT property of $R(C,D)$ via orbifold curves}
Let us now work over a field $k$ of characteristic $p>0$. 
Then we have the Frobenius morphism $F \colon \mathfrak{C} \to \mathfrak{C}$ as in Section \ref{frob}.

Recall that $R=R(C,D)$ is the section ring $R=R(\mathfrak{C},L)=\bigoplus_{\ell \ge0} H^0(\mathfrak{C},L^\ell)t^\ell$ associated with the 
line bundle $L=\mathcal{O}_\mathfrak{C}(\pi^*D)$. 
Then for $0\le i<q=p^e$, the $R$-summand $(R^{1/q})_{i/q  \mod\Z}$ of $R^{1/q}$, which is described as 
$$
(R^{1/q})_{i/q   \mod\Z} 
                    = \bigoplus_{0\le \ell \equiv i   \mod q} H^0(\mathfrak{C},F^e_*(L^\ell))t^{\ell/q} \\
                    \cong \bigoplus_{m\ge 0} H^0(\mathfrak{C},F^e_*(L^i)\otimes L^m)t^{i/q+m},
$$
is equivalent to the $\Z$-graded $R$-module $\Gamma_*(F^e_*(L^i))$ with respect to $\sim$. 
In view of Corollary \ref{1-1cor}, it is important for 
our purpose to know the decomposition of the Frobenius push-forwards $F^e_*(L^i)$ with 
$0\le i\le q-1$ on the orbifold curve $\mathfrak{C}$. 

Given any line bundle $L$ on $\mathfrak{C}$ and integers $e,\,i\ge 0$, we have a unique decomposition 
$$F^e_*(L^i) = \F_1^{(e,i)}\oplus\cdots\oplus\F_{m_{e,i}}^{(e,i)}$$
in $\Coh(\mathfrak{C})$ with $\F_j^{(e,i)}$ indecomposable. 

\begin{definition}
Let $L$ be a line bundle on $\mathfrak{C}$. We say that the pair $(\mathfrak{C},L)$ has \emph{globally finite $F$-representation 
type} $(GFFRT$ for short$)$, if the set of isomorphism classes
$$\{\F_j^{(e,i)} \,|\, e=0,1,2,\dots;\, i=0, 1,\dots,p^e-1;\, j=1,\dots,m_{e,i}\}/_{\displaystyle\cong}$$ 
is finite. 
We say that $\mathfrak{C}$ has GFFRT if the pair $(\mathfrak{C},\mathcal{O}_\mathfrak{C})$ does. 
\end{definition} 

\begin{corollary}\label{FFRT-GFFRT}
Let $R=R(C,D)=R(\mathfrak{C},L)$, where $L=\mathcal{O}_\mathfrak{C}(\pi^*D)$ as above. Then $R$ has FFRT if and only if $(\mathfrak{C},L)$ has 
GFFRT.
\end{corollary}

\begin{proof}
By Corollary \ref{1-1cor}, $R$ has FFRT if and only if the set of equivalence classes 
$$\{\F_j^{(e,i)} \,|\, e=0,1,2,\dots;\, i=0,1,\dots,p^e-1;\, j=1,\dots,m_{e,i}\}/ \sim_L$$ 
is finite. Hence the sufficiency follows immediately. 
For the necessity, it is sufficient to prove the 
following claim. 

\medskip\noindent\textbf{Claim \ref{FFRT-GFFRT}.1.}
For any vector bundle $\F$ on $\mathfrak{C}$, the set 
$$
\left\{m \in \Z \;\begin{array}{|c} \F\otimes L^m \text{ is isomorphic to a direct summand of} \\
                         F^e_*(L^i) \text{ for some } e\ge 0 \text{ and } 0\le i\le p^e-1 \end{array} \right\}
$$
is finite. 

\smallskip
To prove the claim we first note that 
$$
H^0(\mathfrak{C},F^e_*(L^i)\otimes L^{-1}) \cong H^0(\mathfrak{C},L^{i-p^e}) 
                                               = H^0(C,\mathcal{O}_C(\lfloor(i-p^e)D\rfloor)) = 0
$$
for all $e\ge 0$ and $0 \le i\le p^e-1$. Hence, if $\mathcal{F}\otimes L^m$ is a direct
summand of $F^e_*(L^i)$, then we must have 
$H^0(\mathfrak{C},\mathcal{F}\otimes L^{m-1})=0$. On the other hand, there exists 
an integer $m_0$ such that $H^0(\mathfrak{C},\mathcal{F}\otimes L^m) \ne 0$ for 
all $m \ge m_0$. Indeed, if we choose an integer $r>0$ such that $rD$ is integral and
write $m=rs+i$ with $0\le i<r$, then we see that 
$$
H^0(\mathfrak{C},\mathcal{F}\otimes L^m) = H^0(C,\pi_*(\mathcal{F}\otimes L^{i+rs})) 
                  = H^0(C,\pi_*(\mathcal{F}\otimes L^i)\otimes\mathcal{O}_C(rD)^{\otimes s})
$$
is non-zero for $s\gg 0$, since $rD$ is ample. Thus we conclude that if $m \ge m_0+1$, 
then $\mathcal{F}\otimes L^m$ cannot be a direct summand of 
$F^e_*(L^i)$ for any $e\ge 0$ and $1\le i\le p^e-1$. This implies that the set in Claim
\ref{FFRT-GFFRT}.1 is bounded above. Similarly, a dual argument with the first 
cohomology $H^1$ gives a lower bound of the set. 
\end{proof}

Finally we rephrase the $F$-purity of $R$ in terms of the $F$-splitting of an orbifold curve. 
We say that $\mathfrak{C}$ is $F$-split if the Frobenius ring homomorphism $F\colon \mathcal{O}_\mathfrak{C} \to F_*\mathcal{O}_\mathfrak{C}$ splits as an 
$\mathcal{O}_\mathfrak{C}$-module homomorphism, 
and that a ring $R$ is $F$-pure if $\Spec R$ is $F$-split. 
Note that the Frobenius morphism on $\mathfrak{C}$ induces 
$F \colon \omega_\mathfrak{C} \to \omega_\mathfrak{C}\otimes F_*\mathcal{O}_\mathfrak{C} \cong F_*(\omega_\mathfrak{C}^p)$, and that on the first 
cohomology $F \colon H^1(\mathfrak{C},\omega_\mathfrak{C}) \to H^1(\mathfrak{C},F_*(\omega_\mathfrak{C}^p)) \cong H^1(\mathfrak{C},\omega_\mathfrak{C}^p)$ when $p$ does not divide any weight $r_{i}$. 

\begin{proposition}\label{F-split}
In the notation as above, we assume that the characteristic $p$ of the field $k$ does not divide any $r_{i}$.
Then the following three conditions are equivalent.
\begin{enumerate}
\item[$(1)$] $R(C,D) \cong R(\mathfrak{C},L)$ is F-pure.
\item[$(2)$] $\mathfrak{C}$ is F-split. 
\item[$(3)$] The induced Frobenius $F\colon H^1(\mathfrak{C},\omega_\mathfrak{C}) \to  H^1(\mathfrak{C},\omega_\mathfrak{C}^p)$ is injective. 
\end{enumerate}
\end{proposition}

\begin{proof}
Since Serre duality and adjunction formula hold for $\mathfrak{C}$ as we noted in Section \ref{coh} and Section \ref{aff}, usual arguments hold as follows.
First note that $\hom(F_*\mathcal{O}_\mathfrak{C},\omega_\mathfrak{C}) \cong F_*\omega_\mathfrak{C}$ by the 
adjunction formula \eqref{adjunction}, so that 
$\hom(F_*\mathcal{O}_\mathfrak{C},\mathcal{O}_\mathfrak{C}) \cong 
  F_*(\omega_\mathfrak{C})\otimes\omega_{\mathfrak C}^{-1} \cong F_*(\omega_\mathfrak{C}^{1-p})$. 
Now $\mathfrak{C}$ is $F$-split if and only if the dual Frobenius map 
$F^{\vee} \colon \Hom(F_*\mathcal{O}_\mathfrak{C},\mathcal{O}_\mathfrak{C}) \to \Hom(\mathcal{O}_\mathfrak{C},\mathcal{O}_\mathfrak{C})$ 
is surjective. 
Since this map is identified with 
$$
F^{\vee} \colon H^0(\mathfrak{C},F_*(\omega_\mathfrak{C}^{1-p})) \to H^0(\mathfrak{C},\mathcal{O}_\mathfrak{C}), 
$$
which is dual to 
the induced Frobenius 
$F\colon H^1(\mathfrak{C},\omega_\mathfrak{C}) \to  H^1(\mathfrak{C},F_*(\omega_\mathfrak{C}^p))$ 
by the Serre duality, 
the equivalence of (2) and (3) follows. On the other hand, the equivalence of (1) and (3) follows from 
\cite{W}, since Lemma \ref{cadman} allows us to identify the induced Frobenius map in (3) with 
$F \colon H^1(C,\mathcal{O}_C(K_C)) \to H^1(C,\mathcal{O}_C(\lfloor p(K_C+D')\rfloor))$, where $D'$ is the ``fractional 
part" of $D$ so that $\pi^*(K_C+D')=K_\mathfrak{C}$. 
\end{proof}


\section{Weighted projective lines}

In this section, we consider an orbifold curve $\mathfrak{C}=C[\sqrt[r_1]{P_1},\dots,\sqrt[r_n]{P_n}\,]$ with $C=\PP^{1}$, that is, we have a coarse moduli map $\pi \colon \mathfrak{C} \to \PP^{1}$.
In this case, $\mathfrak{C}$ is called a \emph{weighted projective line}. 


\subsection{Homogeneous coordinate ring}
Here we construct $\mathfrak{C}$ as a quotient stack $[U/G]$ following \cite{GL}.
We take the homogeneous coordinate ring $T=k[z_{1}, z_{2}]$ of the projective line $\PP^{1}$ such that $P_{1}=\lbrace z_{1}=0 \rbrace, P_{2}=\lbrace z_{2}=0 \rbrace$, and $P_{i}=\lbrace z_2-\lambda_i z_1=0\rbrace$ 
for $\lambda_{i} \in k$ and $i=3, \ldots, n$. We consider a $T$-algebra 
$$S = T[x_1,\ldots,x_n]/
         \langle x_1^{r_1}-z_1,x_2^{r_2}-z_2,x_3^{r_3}-(z_2-\lambda_3z_1),\ldots,x_n^{r_n}-(z_2-\lambda_nz_1)\rangle,
$$ 
and take an open subset 
$U= \mathop{\mathrm{Spec }}\nolimits S \setminus \lbrace (x_{1}, x_{2})=(0,0) \rbrace$. 
We define a group $G$ acting on $U$ by $G = \Hom_{\mathop{\mathrm{group}}\nolimits}(\Gamma, k^{\ast})$ for 
$$
\Gamma= \bigoplus_{i=1}^{n} \Z \vec{a}_{i} \oplus \Z \vec{c} / \langle r_{i} \vec{a}_{i} - \vec{c} \mid i=1, \ldots, n \rangle,
$$
where $\vec{a}_{1}, \ldots, \vec{a}_{n}, \vec{c}$ are formal basis elements of $\bigoplus_{i=1}^{n} \Z \vec{a}_{i} \oplus \Z \vec{c}$.
In other words, we define $G$ by
$$
G = \Spec k[a_{1}^{\pm 1}, \ldots, a_{n}^{\pm 1}, c^{\pm 1}]/\langle a_{i}^{r_{i}} - c \mid i= 1, \ldots, n \rangle. 
$$ 
Here $G$ acts on $U$ diagonally
$$
(x_{1}, \ldots, x_{n}) \mapsto (a_{1} x_{1}, \ldots, a_{n} x_{n}).
$$
In other words, the $G$-action is given by the $\Gamma$-grading of $S$ defined by $\deg_{\Gamma} x_{i} = \vec{a}_{i}$ for $i= 1, \ldots, n$.
Here we use the fact that giving $G$-actions is equivalent to giving gradings in $\Gamma$, since we have the similar statement as in \cite[Proposition 4.7]{Muk}, which is also proved similarly. 
This action is compatible with the natural morphism $U \to \PP^{1}$ induced by the $T$-algebra structure of $S$, and gives $\pi' \colon [U/G] \to \PP^{1}$.
Furthermore, by $G$-weight spaces $k a_{i}$ with $G$-action given by the multiplication of $a_{i}$, we have line bundles $\mathcal{L}_{i} =[U \times k a_{i} / G]$ on the quotient stack $[U/G]$, and sections $s_{i} \in \Gamma([U/G], \mathcal{L}_{i})$ defined by $x_{i}$.
Since we have natural isomorphisms $\alpha \colon \mathcal{L}_{i}^{\otimes r_{i}} \cong (\pi')^{\ast} \mathcal{O}_{\PP^{1}}(P_{i})$ sending $x_{i}^{r_{i}}$ to $\lambda_{i}z_{1} - z_{2}$, this datum defines a morphism $\varphi \colon [U/G] \to \mathfrak{C}$ such that $\pi'=\pi \circ \varphi$ by Lemma 
\ref{cadman} (3).
By the local description \eqref{local}, we see that $\varphi$ is an isomorphism $[U/G] \cong \mathfrak{C}$.
In the following, we identify $\mathfrak{C}$ with the quotient stack $[U/G]$ via this isomorphism, and we call the $\Gamma$-graded algebra $S$ a \emph{homogeneous coordinate ring} of $\mathfrak{C}$.
We have $Q_{i}=\lbrace x_{i} = 0 \rbrace$ on $\mathfrak{C}=[U/G]$, and $\mathcal{O}_{\mathfrak{C}}(Q_{i}) \cong \mathcal{L}_{i}$.

By this construction, we have an identification $\Pic \mathfrak{C} \cong \Gamma$ sending $\mathcal{O}_{\mathfrak{C}}(Q_{i})$ to $\vec{a}_{i}$.
We write by $\deg \colon \Pic \mathfrak{C} \to \mathbb{Q}$ the map taking degrees of Chern classes of line bundles on $\mathfrak{C}$. 
We have $\deg \vec{a}_{i} = \frac{1}{r_{i}}$ for $i=1, \ldots, n$, and $\deg \vec{c} =1$, and in particular, when $p$ does not divide any weight $r_{i}$, we have $\deg \omega_{\mathfrak{C}} = n-2 - \sum_{i=1}^{n} \frac{1}{r_{i}}$ by Lemma \ref{dualizing}.
We put $\delta_{\mathfrak{C}}=n-2 - \sum_{i=1}^{n} \frac{1}{r_{i}}$.


\subsection{Indecomposable vector bundles on a weighted projective line}
We introduce a classification of indecomposable vector bundles on $\mathfrak{C}$ by \cite{CB}.
We consider a lattice 
$$
\mathfrak{L}= \Z \alpha_{\ast} \oplus \bigoplus_{i=1}^{n} \bigoplus_{j=1}^{r_{i}-1} \Z \alpha_{ij}, 
$$
and put $\hat{\mathfrak{L}}= \mathfrak{L} \oplus \Z \delta$.
Here $\alpha_{\ast}, \alpha_{ij}$ corresponds to the following graph consisting of vertices $\ast, ij$ for $i=1, \ldots, n$ and $j=1, \ldots, r_{i} -1$, and edges joining $\ast$ and $i1$, and $ij$ and $i j+1$, for $i=1, \ldots,n, j = 1, \ldots, r_{i}-2$.
The \emph{Cartan matrix} is defined by $C= 2E - A$, where $E$ is the identity matrix and $A$ is the adjacency matrix of the above graph.
This defines an inner product $( , ) \colon \mathfrak{ L} \times \mathfrak{L} \to \mathfrak{L}$.

We define the set $\Pi= \lbrace \alpha_{\ast} \rbrace \cup \lbrace \alpha_{ij} \mid i =1, \ldots, n, j = 1, \ldots, r_{i}-1 \rbrace$ of simple roots as follows.
An element $\alpha \in \Pi$ defines a reflection $\mathfrak{L} \to \mathfrak{L}$; $\lambda\mapsto\lambda-(\alpha,\lambda)\alpha$.
We define the \emph{Weyl group} $W$ by the subgroup of $\mathop{\mathrm{Aut}}\nolimits\, \mathfrak{L}$ generated by these reflections, and put $\Delta^{\mathop{\mathrm{re}}\nolimits} = W \Pi$. 
We define the fundamental set 
$$
M=\lbrace v \in \mathfrak{L}^{+} \mid v \neq 0, (\alpha, v) \le 0 \mathop{\mathrm{ for }}\nolimits
\alpha \in \Pi, \mathop{\mathrm{ support }}\nolimits \mathop{\mathrm{ of }}\nolimits v \mathop{\mathrm{ is }}\nolimits \mathop{\mathrm{ connected }}\nolimits \rbrace
$$
where $\mathfrak{L}^{+}=\lbrace v=l_{\ast} \alpha_{\ast} + \sum l_{ij} \alpha_{ij} \in \mathfrak{L} \mid l_{\ast}, l_{ij} \ge 0 \rbrace$.
We put $\Delta=\Delta^{\mathop{\mathrm{re}}\nolimits} \cup W M \cup W(-M)$, and $\hat{\Delta} = \lbrace \alpha + m  \delta \in \hat{\mathfrak{L}} \mid \alpha \in \Delta, m \in \Z \rbrace \cup \lbrace m \delta \in \hat{\mathfrak{L}} \mid m \in \Z , m \neq 0\rbrace$.
An element $v$ of $\hat{ \Delta}$ is called a \emph{root}.
It is called a \emph{real root}, if $v=\alpha + m \delta$ for $\alpha \in \Delta^{re}$, otherwise it is called \emph{imaginary root}.

For a vector bundle $\E$ on $\mathfrak{C}$, we associate a vector bundle $\F=\pi_{\ast} \E$ on $\PP^{1}$ and flags 
$$
\lbrace 0=\F_{ir_{i}} \subset \cdots \subset \F_{ij} \subset \cdots\subset \F_{i0} =\F|_{P_{i}} \rbrace_{i=1}^{n}, 
$$ 
where we define $\F_{ij}$ by the image of $\pi_{\ast} \E(- j Q_{i}) |_{P_{i}} \to \pi_{\ast} \E|_{P_{i}}=\F_{i0}$ for $i=1, \ldots, n, j= 1, \ldots, r_{i}$.
We note that $\F_{ir_{i}}=0$ follows from $r_{i} Q_{i} \cong \pi^{\ast} P_{i}$ and the projection formula.
We define a \emph{type} $t(\E) \in \hat{\mathfrak{L}}$ of $\E$ by 
$$
t(\E) = ( \rk \F ) \alpha_{\ast} + \sum_{i=1}^{n} \sum_{j=1}^{r_{i}-1} ( \dim \F_{ij} ) \alpha_{ij} + (\deg \F) \delta.
$$
This defines a map $t \colon K(\mathfrak{C}) \to \hat{\mathfrak{L}}$ from the Grothendieck group $K(\mathfrak{C})$ of $\mathfrak{C}$.
We consider the subset $\hat{\mathfrak{L}}^{+} \subset \hat{\mathfrak{L}}$ of positive linear combinations of $\alpha_{\ast} + m \delta$, $\delta$, $\alpha_{ij}$ and $-\sum_{j=1}^{r_{i}-1} \alpha_{ij} + \delta$ for $m \in \Z$ and $i=1, \ldots, n$.
We call elements in $\hat{\Delta} \cap \hat{\mathfrak{L}}^{+}$ \emph{positive roots}.

The following is due to \cite[Theorem 1]{CB}.

\begin{theorem}[Crawley-Boevey]
\label{indec}
For an element $t \in \hat{\mathfrak{L}}$, there exists an indecomposable sheaf on $\mathfrak{C}$ with the type $t$ if and only if 
$t$ is a positive root. There is  a unique isomorphism class of indecomposable sheaf for a real root, infinitely many for an imaginary root.
\end{theorem}

We take an ample Cartier divisor $D=\sum_{i=1}^{n} \frac{s_{i}}{r_{i}} P_{i}$ on $\PP^{1}$, and put $L = \pi^{\ast} \mathcal{O}_{\PP^{1}}( D)$.
Combining Theorem \ref{indec} with Corollary \ref{FFRT-GFFRT}, we have the following:

\begin{theorem}\label{FRT}
For a weighted projective line $\mathfrak{C}$, the set of equivalence classes of indecomposable vector bundles with respect to $\sim_{L}$ is finite, if and only if $\delta_{\mathfrak{C}} < 0$.
In this case, the graded ring $R$ has FFRT.
\end{theorem}

\proof
By direct computations, we see that $\delta_{\mathfrak{C}} < 0$ if and only if the corresponding graphs are of 
finite type. 
It is equivalent to saying that $\Delta$ is equal to $\Delta^{\mathop{\mathrm{re}}\nolimits}$ as in \cite[Chapter I]{K}. 
It is also known that in this case $\Delta=\Delta^{\mathop{\mathrm{re}}\nolimits}$ is finite.

Hence if $\delta_{\mathfrak{C}} <0$, then ranks of indecomposable vector bundles are bounded.
We put
$$
r_{\max}=\max \lbrace \rk \E \mid \E \mathop{\mathrm{ indecomposable }}\nolimits \mathop{\mathrm{ vector }}\nolimits \mathop{\mathrm{ bundle }}\nolimits \mathop{\mathrm{ on }}\nolimits \mathfrak{C} \rbrace.
$$
We have $L^{\otimes r_{1} \cdots r_{n} } = \mathcal{O}_{\mathfrak{C}}(r_{1} \cdots r_{n} \pi^{\ast} D) \cong \pi^{\ast} \mathcal{O}_{\PP^{1}}(d)$, where $d=r_{1} \cdots r_{n} \sum_{i=1}^{n} \frac{s_{i}}{r_{i}}$ is a positive integer since $D$ is ample.
Then after tensoring $L$ suitably many times, the coefficient of $\delta$ in the type of any indecomposable vector bundle lies between $0$ and $d r_{\max}$.
Hence the set of equivalence classes is finite by Theorem \ref{indec}.

On the other hand, if $\delta_{\mathfrak{C}} \ge 0$, then we have infinitely many isomorphism classes of indecomposable vector bundles whose type is a fixed imaginary root.
Since tensoring $L$ changes types, these vector bundles are not equivalent to each other with respect $\sim_{L}$. 

Finally from the description of grading structure of $R^{\frac{1}{q}}$ below Definition \ref{ffrt}, we see that the last statement follows from Corollary \ref{1-1cor}.
\endproof

\begin{remark}
As a special case of $\delta_{\mathfrak{C}}<0$, we have the toric case, in which the 
weighted projective line $\mathfrak{C}$ has at most two stacky points. In this case, for every line bundle $L$ on 
$\mathfrak{C}$, the Frobenius push-forward $F^e_*L$ is decomposed into direct sum of line bundles (cf.\ 
\cite[Theorem 4.5]{OU}). 
\end{remark}


\section{Frobenius summands on weighted projective lines with $\delta_{\mathfrak{C}}=0$}

In this section we study the structure of the Frobenius push-forward $F^e_*\mathcal{O}_{\mathfrak{C}}$ 
on a weighted projective line $\mathfrak{C}$ when $\delta_{\mathfrak{C}}$ is equal to 0. We assume 
that $\mathfrak{C}$ has $n$ stacky points $Q_1,\dots,Q_n$ of weights $r_1,\dots,r_n$ lying over 
$\lambda_1,\dots,\lambda_n \in \PP^1$, respectively.
We also assume that $p$ does not divide any weight $r_i$ and 
that the weights are ordered as $r_1\le\cdots\le r_n$. Since the assumption $\delta_{\mathfrak{C}}=0$ 
is equivalent to $\sum_{i=1}^{n} \frac{r_i-1}{r_i}=2$, it follows that $n=3$ or $4$ and the weight 
$(r_1,\dots,r_n)$ is either one of the following: $(2,3,6)$, $(2,4,4)$, $(3,3,3)$, $(2,2,2,2)$. Also the 
canonical bundle $\omega_{\mathfrak{C}}$ is torsion of order $m:=\mathrm{lcm}\{r_1,\dots,r_n\}=r_n$. 

\begin{lemma}\label{r-cover}
Let $\mathfrak{C}$ be a weighted projective line with $\delta_{\mathfrak{C}}=0$ as above, and suppose that the characteristic $p$ of $k$ does not divide $m=r_n$. Then there exists an elliptic curve $E$ with $\mu_m\cong\Z/m\Z$-action and 
an $m$-fold covering $f \colon E \to \PP^1$ which factors through $\mathfrak{C}$ as 
$$f=\pi\circ\varphi \colon E \stackrel{\varphi}{\lra} \mathfrak{C} \stackrel{\pi}{\lra} \PP^1,$$
satisfying the following conditions.
\begin{enumerate}
\item[$(1)$] $\mathfrak{C}=[E/\mu_m]$ and $\PP^1=E/\mu_m$ via $\varphi$ and $f$, respectively.
\item[$(2)$] $\varphi$ is unramified and $\varphi_*\mathcal{O}_E \cong \bigoplus_{\ell=0}^{m-1} \omega_\mathfrak{C}^{\otimes(-\ell)}$.
\item[$(3)$] There exist exactly $m/r_i$ points of $E$ lying over the stacky point $Q_i$ whose ramification 
index with respect to $f$ is equal to $r_i$.
\item[$(4)$] Choose the point $P_n \in E$ lying over $Q_n$ as the zero element of $E$ as a group. 
If $P \in E$ is a ramification point of $f$ lying over one of the stacky points $Q_i$, then $P$ is an 
$m$-torsion point with respect to the group law of $(E,P_n)$. 
\item[$(5)$] $\mathfrak{C}$ is F-split if and only if $E$ is ordinary (or equivalently, F-split), and in this case, 
$p \equiv 1 \pmod m$.  
\end{enumerate}
\end{lemma}

\begin{proof}
Let $\mathcal{A}=\bigoplus_{\ell\in\Z_m} \omega_\mathfrak{C}^{\otimes(-\ell)}$ with an $\mathcal{O}_\mathfrak{C}$-algebra structure defined by $\omega_\mathfrak{C}^{\otimes(-m)} \cong \mathcal{O}_\mathfrak{C}$ and let $\varphi \colon E=\mathrm{Spec}_\mathfrak{C}\mathcal{A} \to \mathfrak{C}$ be the induced morphism.  
We recall the description of $\mathfrak{C}=[U/G]$ in subsection 4.1. 
Then the $\mathcal{O}_\mathfrak{C}$-algebra $\mathcal A$ corresponds to the $\Gamma$-graded $S$-algebra $S[\xi]/\langle\xi^m-1\rangle$. 
Here $\deg_{\Gamma} \xi=-\deg_{\Gamma} \omega_\mathfrak{C}=-(n-2)\vec{c}+\sum_{i=1}^n\vec{a}_i \in \Gamma$. 

By local computations, we see that every closed point in $E$ has a trivial automorphism functor. 
Hence by \cite[Theorem 2.2.5]{C}, $E$ is an algebraic space, and the coarse moduli map $E \to \mathrm{Spec}_{\PP^1}(\pi_*\mathcal{A})$ is an isomorphism. 
Thus $f=\pi\circ\varphi$ 
is identified with the structure morphism $\mathrm{Spec}_{\PP^1}(\pi_*\mathcal{A}) \to \PP^1$. 


$$
\xymatrix{
E \ar[d]_{\cong} \ar[r]^{\varphi} \ar[rd]^{f} & \mathfrak{C} \ar[d]^{\pi} \\
\mathrm{Spec}_{\PP^1}(\pi_*\mathcal{A}) \ar[r] & \mathbb{P}^{1} 
}
$$
\noindent 
We have a $\mu_m$-action on $E=\mathrm{Spec}_\mathfrak{C}\mathcal{A}$ by $\deg\xi=1 \in \Z/m\Z=(\mu_m)^{\vee}$. 
This gives a proof of (1) and (2). 

We prove (3) and (4) examining the $m$-fold covering 
$f \colon E \cong \mathrm{Spec}_{\PP^1}(\pi_*\mathcal{A}) \to \PP^1$ 
case by case for each weight. Then it follows that $E$ is an elliptic curve from Hurwitz's formula.  

The cases for weight $(3,3,3)$, $(2,2,2,2)$ are easy: Indeed, $f$ is totally ramified at all the ramification points 
$P_1,\dots,P_n$ in these cases, so that $mP_i=f^*(\lambda_i) \sim f^*(\lambda_j)=mP_j$ for $1\le i,j \le n$. 
Hence $\mathcal{O}_E(P_i-P_n)^{\otimes m}\cong\mathcal{O}_E$ for all $i$, which means that $P_i$ is $m$-torsion.

As for weight $(2,4,4)$, we choose $\lambda_1=0$, $\lambda_2=1$, $\lambda_3=-1 \in \PP^1$ in an 
affine coordinate $u$ of $\PP^1$ and a $\Q$-divisor $B=-\frac12(0)+\frac14(1)+\frac14(-1) $ on $\PP^1$ 
so that $\pi^*B \sim -K_\mathfrak{C}$. We give an $\mathcal{O}_{\PP^1}$-algebra structure of 
$\pi_*\mathcal{A} \cong \bigoplus_{\ell=0}^3 \mathcal{O}_{\PP^1}(\lfloor\ell{B}\rfloor)z^{\ell}$ 
by the isomorphism $\mathcal{O}_{\PP^1}(4B)z^4 \cong \mathcal{O}_{\PP^1}$ via which $\frac{u^2}{u^2-1}z^4$ corresponds 
to $1$. Then for an affine open subset $U=\Spec k[u]$ of $\PP^1$, 
$H^0(U,\pi_*\mathcal{A}) \cong k[u,uz,uz^2]/\langle u^2z^4-u^2+1\rangle$, 
and the $4$-fold covering $f$ locally looks like  
$$f \colon f^{-1}U \cong \Spec k[u,v,w]/\langle w^2-u^2+1,uw-v^2\rangle \to U=\Spec k[u].$$
It follows that $f$ has four ramification points $P_1,P_2,P_3,P_4$ with ramification indices 2, 2, 4, 4 whose 
affine coordinates with respect to $u,v,w$ are $(0,0,\sqrt{-1}),\,(0,0,-\sqrt{-1}),\,(1,0,0)$, $(-1,0,0)$,
respectively. Clearly $4P_3\sim4P_4$ and $P_3,P_4$ are $4$-torsion points with respect to the group 
law of $(E,P_4)$. On the other hand, choosing $\varphi=(w-\sqrt{-1})/(u+1)\in k(E)$, we see that 
$\mathrm{div}_E(\varphi)=4P_1-4P_4$, so that $4P_1\sim4P_4$. Similarly, $4P_2\sim4P_4$ and we 
see that $P_1,P_2$ are also $4$-torsion. Thus (3) and (4) are proved for weight $(2,4,4)$. 

The case for weight $(2,3,6)$ is proved similarly. 
We choose $\lambda_1=\infty$, $\lambda_2=0$, $\lambda_3=1 \in \PP^1$ in an affine coordinate $x$ in $\PP^1$ and a $\Q$-divisor $B=-\frac12(\infty)+\frac13(0)+\frac16(1)$ on $\PP^1$ so that $\pi^*B \sim -K_{\mathfrak{C}}$.  
We give an $\mathcal O_{\PP^1}$-algebra structure of 
$\pi_*\mathcal{A} \cong \bigoplus_{\ell=0}^5 \mathcal O_{\PP^1}(\lfloor\ell{B}\rfloor)z^{\ell}$ 
by the isomorphism $\mathcal O_{\PP^1}(6B)z^6 \cong \mathcal O_{\PP^1}$ via which $\dfrac{z^6}{x^2(x-1)}$ 
corresponds to $1$. Let $t=1/x$ and take affine open neighborhoods $U=\Spec k[x]$ of $0 \in \PP^1$ 
and $V=\Spec k[t]$ of $\infty \in \PP^1$, respectively. Then the ring $H^0(U,\pi_*\mathcal{A})$ (resp.\
$H^0(V,\pi_*\mathcal{A})$) is isomorphic to \\
$k[x,z,z^3/x]/\langle (z^3/x)^2-x+1\rangle 
                            \cong k[x,y,z]/\langle xy-z^3,y^2-x+1\rangle \cong k[y,z]/\langle y^3+y-z^3\rangle$
(resp.\
$k[t,tz,tz^2]/\langle t^3z^6+t-1\rangle 
                            \cong k[t,u,v]/\langle v^3+t-1,u^2-tv\rangle \cong k[u,v]/\langle u^2+v^4-v\rangle$). 
Thus the $6$-fold covering $f$ locally looks like 
$$f|_{f^{-1}U} \colon f^{-1}U \cong \Spec k[y,z]/\langle y^3+y-z^3\rangle \to U=\Spec k[y^2+1]$$
and
$$f|_{f^{-1}V} \colon f^{-1}V \cong \Spec k[u,v]/\langle u^2+v^4-v\rangle \to V=\Spec k[v^3-1],$$
respectively. We see that $f$ has six ramification points $P_1,\dots,P_6$, the $i$-th one $P_i$ of which 
is described as follows, where $\omega \in k$ denotes a primitive cube root of unity. 
\begin{itemize}
\item $i=1,2,3$: $P_i$ is a ramification point of index $2$ at $(u,v)=(0,\omega^i)$ over $\infty \in \PP^1$; 
\item $i=4,5$: $P_i$ is a ramification point of index $3$ at $(y,z)=(\sqrt{-1}^{2i+1},0)$ over $0 \in \PP^1$; 
\item $P_6$ is a ramification point of index 6  at $(y,z)=(0,0)$ ($(u,v)=(0,0)$) over $1 \in \PP^1$. 
\end{itemize}
Choosing rational function $\varphi_i \in k(E)$ given by 
$$\varphi_i=\frac{v-\omega^i}{v}\quad(i=1,2,3)\quad\text{ and }\quad
   \varphi_i=\frac{y-\sqrt{-1}^{2i+1}}{y}\quad(i=4,5),$$
we see that $\mathrm{div}_E(\varphi_i)=2P_i-2P_6$ for $i=1,2,3$ and $\mathrm{div}_E(\varphi_i)=3P_i-3P_6$ 
for $i=4,5$, respectively. Thus the points $P_1,\dots,P_6$ are $6$-torsion with respect to the group law of 
$(E,P_6)$. 

To prove (5) recall that the elliptic curve $E$ is ordinary if and only if the Frobenius 
$$F \colon H^1(E,\mathcal{O}_E) \to H^1(E,F_*\mathcal{O}_E) \cong H^1(E,\mathcal{O}_E)$$
is injective(cf.\ Proposition \ref{F-split}). Since $\varphi_*\mathcal{O}_E\cong\bigoplus_{\ell\in\Z_m}\omega_\mathfrak{C}^{ \otimes (-\ell)}$, 
this is equivalent to the injectivity of 
$F \colon H^1(\mathfrak{C},\omega_\mathfrak{C}^{ \otimes (-\ell)}) \to H^1(\mathfrak{C},\omega_\mathfrak{C}^{ \otimes (-\ell p)})$ for all $\ell \in \Z_m$. 
Thus $\mathfrak{C}$ is $F$-split if so is $E$, by Proposition \ref{F-split}. Conversely, if $\mathfrak{C}$ is $F$-split, then we 
must have $H^1(\mathfrak{C},\omega_\mathfrak{C}^p) \ne 0$. Since $\omega_\mathfrak{C}$ is an $m$-torsion line bundle, this implies that 
$p \equiv 1 (\mod m)$ and the Frobenius on $H^1(E,\mathcal{O}_E) \cong k$ is identified with 
$F \colon H^1(\mathfrak{C},\omega_\mathfrak{C}) \to H^1(\mathfrak{C},\omega_\mathfrak{C}^p)$. Therefore the $F$-splitting of $\mathfrak{C}$ implies that $E$ 
is ordinary.  
\end{proof}

\begin{remark}
Lemma \ref{r-cover} (5) is also verified with explicit computations of the induced Frobenius map 
$$
F \colon H^1(\mathfrak{C},\omega_\mathfrak{C}) \to H^1(\mathfrak{C},\omega_\mathfrak{C}^p)
$$ 
and Fedder's criterion \cite{F} applied to the defining equation of $E$. 
\end{remark}

Now we state the main result of this section. 

\begin{theorem}\label{genus1}
Let $\mathfrak{C}$ be a weighted projective line with $\delta_{\mathfrak{C}}=0$, and assume that the characteristic $p$ 
does not divide any weight $r_i$. Then $\mathfrak{C}$ does not have GFFRT.
\end{theorem}
\begin{proof}
Let $f \colon E \to \PP^1$ be the $m$-fold covering from an elliptic curve constructed in Lemma 
\ref{r-cover} and let $\varphi \colon E \to \mathfrak{C}$ be the induced morphism. We divide the proof into 
two cases, according to whether $E$ is ordinary or supersingular. First we recall the following:

\begin{lemma}[\cite{A}, {\cite[Lemma 4.12]{HSY}}]\label{elliptic}
Let $E$ be an elliptic curve in characteristic $p$ and let $q=p^e$ for $e\ge 0$. 
\begin{enumerate}
\item[$(1)$] If $E$ is ordinary, then $F^e_*\mathcal{O}_E$ splits into $q$ distinct $q$-torsion line bundles. 
\item[$(2)$] If $E$ is supersingular, then $F^e_*\mathcal{O}_E$ is isomorphic to Atiyah's vector bundle $\F_q$ of 
rank $q$ (see subsection 5.1 below).
\end{enumerate}
\end{lemma}


\subsection{Supersingular case}

On the elliptic curve $E$ (which we do not yet assume to be supersingular), we have indecomposable 
vector bundles $\F_r$ of rank $r$ and degree $0$ such that $H^{0}(E, \F_{r}) \cong k$ for all integer 
$r > 0$. This bundle is determined inductively by $\F_1=\mathcal{O}_E$ and a unique non-trivial extension 
\begin{eqnarray}\label{ex0}
0 \to \mathcal{O}_E \to \F_r \to \F_{r-1} \to 0
\end{eqnarray}
as in \cite[Theorem 5]{A}.
In what follows, we construct inductively vector bundles $\G_r$ on $\mathfrak{C}$ of rank $r=1,2,\dots$ such that 
\begin{eqnarray}
\label{ext}
\Ext_{\mathfrak{C}}^1(\G_r,\omega_\mathfrak{C}^i) \cong 
\left\{ \begin{array}{cc}
k& \mathop{\mathrm{ if }}\nolimits i \equiv 1 \pmod{m}\\
0& \mathop{\mathrm{otherwise.}}\nolimits
\end{array}
\right.
\end{eqnarray}

We put $\G_{1}= \mathcal{O}_{\mathfrak{C}}$. 
Then we can easily verify condition \eqref{ext} for $r=1$ by computing  
$\Ext_{\mathfrak{C}}^1(\G_1,\omega_\mathfrak{C}^i) \cong H^1(\mathfrak{C},\omega_\mathfrak{C}^i)$ with Lemma \ref{cadman} (4) and Lemma \ref{dualizing}.  

Now let $r\ge 2$ and assume condition \eqref{ext} for $r-1$. Since $\Ext^1(\G_{r-1},\omega_\mathfrak{C}) \cong k$, 
we have a vector bundle $\G_r$ sitting in a unique non-trivial extension 
\begin{eqnarray}
\label{ex}
0 \to \mathcal{O}_\mathfrak{C} \to \G_r \to \G_{r-1} \otimes \omega_\mathfrak{C}^{-1} \to 0.
\end{eqnarray}
We apply the functor $\Ext(-,\omega_\mathfrak{C}^i)$ to this exact sequence to verify condition \eqref{ext}. 
For $i=0$ we have an exact sequence 
$$\Hom(\mathcal{O}_\mathfrak{C},\mathcal{O}_\mathfrak{C}) \stackrel{\delta}{\lra} 
   \Ext^1(\G_{r-1},\omega_\mathfrak{C}) \to \Ext^1(\G_r,\mathcal{O}_\mathfrak{C}) \to \Ext^1(\mathcal{O}_\mathfrak{C},\mathcal{O}_\mathfrak{C})=0,
$$
where the connecting homomorphism $\delta$ is an isomorphism by the non-triviality of the 
extension (\ref{ex}). Thus we have $\Ext^1(\G_r,\mathcal{O}_\mathfrak{C})=0$. 
For $i \not\equiv 0$ (mod $m$), 
we have
$$0=\Ext^1(\G_{r-1},\omega_\mathfrak{C}^{i+1}) \to \Ext^1(\G_r,\omega_\mathfrak{C}^i) \to \Ext^1(\mathcal{O}_\mathfrak{C},\omega_\mathfrak{C}^i) \to 0
$$
by induction, so that $\Ext_\mathfrak{C}^1(\G_r,\omega_\mathfrak{C}^i) \cong \Ext^1(\mathcal{O}_\mathfrak{C},\omega_\mathfrak{C}^i)$. 
Thus condition \eqref{ext} holds for $r$. 

\begin{proposition}\label{atiyah}
Suppose that $m$ is not divisible by $p$. Then $\varphi^*\G_r \cong \F_r$. In particular, $\G_r$ is indecomposable. 
\end{proposition}

\begin{proof}
The assertion is clear if $r=1$. Let $r \ge 2$ and let the exact sequence (\ref{ex}) be given by a non-zero 
extension class $\varepsilon \in \Ext^1(\G_{r-1},\omega_\mathfrak{C}) \cong H^1(\mathfrak{C},\G_{r-1}^{\vee}(K_\mathfrak{C}))$. Since 
$\varphi^*\omega_\mathfrak{C} \cong \mathcal{O}_E$ (note that $\varphi$ is \'etale) and $\varphi^*\G_{r-1} \cong \F_{r-1}$ 
by induction, the pull-back of sequence (\ref{ex}) under $\varphi$ turns out to be 
\begin{eqnarray}\label{ex2}
0 \to \mathcal{O}_E \to \varphi^*\G_r \to \F_{r-1} \to 0.
\end{eqnarray}
This extension is given by the image $\varphi^*\varepsilon$ of $\varepsilon$ under the natural map 
$$\varphi^* \colon 
\Ext^1(\G_{r-1},\omega_\mathfrak{C}) \to \Ext^1(\varphi^*\G_{r-1},\varphi^*\omega_\mathfrak{C}) \cong \Ext^1(\F_{r-1},\mathcal{O}_E).
$$ 
This map is injective, since it is identified with the map 
$$H^1(\mathfrak{C},\G_{r-1}^{\vee}(K_\mathfrak{C})) \to H^1(\mathfrak{C},\G_{r-1}^{\vee}(K_\mathfrak{C})\otimes \varphi_*\mathcal{O}_E)) 
                                         \cong H^1(E,\varphi^*(\G_{r-1}^{\vee}(K_\mathfrak{C})))$$
induced by the splitting map $\mathcal{O}_\mathfrak{C} \to \varphi_*\mathcal{O}_E \,(= \bigoplus_{l=0}^{m-1} \omega_\mathfrak{C}^{-l})$. Thus 
$\varphi^*\varepsilon\ne0$ and it sits in $\Ext^1(\F_{r-1},\mathcal{O}_E) \cong k$. Comparing extensions 
\eqref{ex0} and \eqref{ex2} we see that $\F_r \cong \varphi^*\G_r$, as required.
\end{proof}

We now consider the case where $\mathfrak{C}$ is not $F$-split, or equivalently, $E$ is supersingular.

\begin{proposition}\label{2mod3}
Under the hypothesis of Proposition \ref{atiyah}, assume further that $E$ is supersingular. 
Then we have $F_{\ast}^{e} \mathcal{O}_{\mathfrak{C}} \cong \G_{q}$, where $q=p^{e}$.
\end{proposition}

\begin{proof}
Recall that $\mathcal{O}_\mathfrak{C}$ is a direct summand of 
$\varphi_*\mathcal{O}_E  \,(= \bigoplus_{l=0}^{m-1} \omega_\mathfrak{C}^{-l})$. Then $F^e_*\mathcal{O}_\mathfrak{C}$ is 
a direct summand of $F^e_*\varphi_*\mathcal{O}_E =\varphi_*F^e_*\mathcal{O}_E$. On the other hand, it follows from 
Proposition \ref{atiyah} and Lemma \ref{elliptic} that $\G_q$ is a direct summand of 
$\G_q\otimes \varphi_*\mathcal{O}_E \cong \varphi_*\varphi^*\G_q \cong \varphi_*\F_q \cong \varphi_*F^e_*\mathcal{O}_E$. 
Thus, both $\G_q$ and $F^e_*\mathcal{O}_\mathfrak{C}$ are direct summands of $\varphi_*F^e_*\mathcal{O}_E$, and 
$h^0(\G_q)=h^0(F^e_*\mathcal{O}_\mathfrak{C})=h^0(\varphi_*F^e_*\mathcal{O}_E)=1$. It then follows from the 
indecomposability of $\G_q$ that it is a direct summand of $F^e_*\mathcal{O}_\mathfrak{C}$.  
However, since $\G_q$ and $F^e_*\mathcal{O}_\mathfrak{C}$ have the same rank $q=p^e$, we conclude that 
$\G_q \cong F^e_*\mathcal{O}_\mathfrak{C}$.
\end{proof}

It immediately follows from the proposition that $\mathfrak{C}$ is not GFFRT in the supersingular case. 


\subsection{Ordinary case.}
We now consider the case where the weighted projective line $\mathfrak{C}$ with $\delta_{\mathfrak{C}}=0$ is $F$-split. In 
this case, $p \equiv 1$ (mod $m$) and we have an $m$-fold covering $f \colon E \to \PP^1$ from 
an ordinary elliptic curve $E$. 
Recall that $f$ factors as
$$f \colon E \stackrel{\varphi}{\lra} \mathfrak{C} \stackrel{\pi}{\lra} \PP^1,$$
where $\varphi \colon E \to \mathfrak{C}$ is unramified and $\pi \colon \mathfrak{C} \to \PP^1$ is the coarse moduli map. 
There is a ramification point $P_0 \in E$ of $f$ with ramification index $m$. We choose $P_0$ 
as the identity point for the group structure of $E$. 
Since $E$ is an ordinary elliptic curve, for any $q=p^e$ there exists exactly $q$ distinct $q$-torsion 
points $P_0,P_{1/q},\dots,P_{(q-1)/q} \in E$, among which $P_{1/q},\dots,P_{(q-1)/q}$ are not ramification 
points of $f$ by Lemma \ref{r-cover} (3). By Lemma \ref{elliptic} the $e$-th Frobenius push-forward 
$F^e_*\mathcal{O}_E$ on $E$ splits into $q$ non-isomorphic $q$-torsion line bundles $L_i=\mathcal{O}_E(P_{i/q}-P_0)$ 
with $i=0,1,\dots,q-1$. Thus we have the following decomposition 
\begin{eqnarray}\label{ordinary1}
\varphi_*F^e_*\mathcal{O}_E \cong \varphi_*\mathcal{O}_E \oplus \varphi_*L_1 \oplus\cdots\oplus \varphi_*L_{q-1}
\end{eqnarray}
into rank $m$ bundles $\varphi_*L_i$. On the other hand, we have 
$\varphi_*\mathcal{O}_E \cong \mathcal{O}_\mathfrak{C} \oplus \omega_\mathfrak{C}^{-1} \oplus\cdots\oplus \omega_\mathfrak{C}^{1-m}$ by Lemma 
\ref{r-cover}, so that 
\begin{eqnarray}\label{ordinary2}
\varphi_*F^e_*\mathcal{O}_E \cong 
                F^e_*\mathcal{O}_\mathfrak{C} \oplus F^e_*(\omega_\mathfrak{C}^{-1}) \oplus\cdots\oplus F^e_*(\omega_\mathfrak{C}^{1-m}).
\end{eqnarray}

The group $G=\mu_m$ is the Galois group of the $m$-fold Galois covering $f \colon E \to \PP^1$. If 
we define the equivalence $\sim$ by $L_i \sim L_j$ if and only if $L_i\cong\sigma^*L_j$ for some $\sigma\in\mu_m$, 
then the line bundles $\mathcal{O}_E=L_0,L_1,\dots,L_{q-1}$ are divided into $r+1$ equivalence classes, where 
$r=\frac{q-1}{m}$. Re-numbering the line bundles, we may and will assume that 
the complete representatives are $\mathcal{O}_E=L_0,L_1,\dots,L_r$. Under this notation we have the following:

\begin{proposition}\label{1mod3}
Let the notation be as above. Then $\varphi_*L_i$ is an indecomposable bundle for $1\le i\le q-1$, 
and $\varphi_*L_i \cong \varphi_*L_j$ if and only if $L_i \cong \sigma^*L_j$ for some $\sigma\in\mu_m$. 
We then have a decomposition 
$$F^e_*(\omega_\mathfrak{C}^i) \cong \omega_\mathfrak{C}^i \oplus \varphi_*L_1 \oplus\cdots\oplus \varphi_*L_r$$
into $r+1$ non-isomorphic indecomposable bundles for $i \in \Z/m\Z$. 
\end{proposition}


We need the following lemma to prove the proposition above. 

\begin{lemma}\label{Oda}
Let $L,M$ be $q$-torsion line bundles on $E$ with $L$ non-trivial. Then 
$$\Hom_{\mathcal{O}_\mathfrak{C}}(\varphi_*L,\varphi_*M) \cong 
   \left\{\begin{array}{ll} k & \mathop{\mathrm{if }}\nolimits \sigma^*L \cong M \mathop{\mathrm{ for\; some }}\nolimits \sigma\in \mu_m \\
                                0 & \mathop{\mathrm{otherwise}}\nolimits. \end{array}\right.
$$
\end{lemma}

\begin{proof}
The proof goes along the same line as Oda's \cite[Section 1, p.\ 43--47]{Od}. First, we have the 
Cartesian diagram


$$
\xymatrix{
\ar[d]_{p_{1}}   E \times \mu_m \ar[r]^{\mu} &  E \ar[d]^{\varphi}\\
   E \ar[r]^{\varphi} & \mathfrak{C} 
}
$$
\noindent 
where $\mu\colon E\times\mu_m \to E$ is the map induced by the action of 
$\mu_m=\Spec k [\xi]/\langle\xi^m-1\rangle$ on $E$, and $p_1$ is the projection. 
This follows from the fact that $[E/\mu_m] \cong \mathfrak{C}$ via $\varphi$ (see Lemma \ref{r-cover} (1)) and 
\cite[(7. 21)]{V}. Since $G=\mu_m$ is finite, $\varphi$ is 
affine, so that $\varphi^*\varphi_*L \cong p_{1*}\mu^*L$. Hence by the adjointness of $\varphi^*$ 
and $\varphi_*$ we obtain 
\begin{eqnarray*}
\Hom_{\mathcal{O}_\mathfrak{C}}(\varphi_*L,\varphi_*M) 
& \cong & \Hom_{\mathcal{O}_E}(\varphi^*\varphi_*L,M) \\
& \cong & \Hom_{\mathcal{O}_E}(p_{1*}\mu^*L,M) \cong H^0(E,(p_{1*}\mu^*(L)\otimes M^{-1})^{\vee}).
\end{eqnarray*}
By the Serre duality this is dual to
$$
H^1(E,p_{1*}\mu^*(L)\otimes M^{-1}) \cong H^1(E,p_{1*}(\mu^*L\otimes p_1^*M^{-1}))
                                                  \cong H^1(E\times G,\mu^*L\otimes p_1^*M^{-1}).
$$
Now let $\lambda \colon G \to \hat{E}=\mathrm{Pic}^{\circ}(E)$ be the morphism sending 
$\sigma \in G$ to $\sigma^*(L)\otimes L^{-1}$, which is injective by our assumption that $L$ is a 
non-trivial $q$-torsion line bundle (cf.\ Lemma \ref{r-cover} (3)). 
Then the restriction of the line bundle 
$\mu^*L\otimes p_1^*L^{-1}$ on $E\times G$ to $\{P_0\}\times G$ is trivial and its restriction to 
$E\times\{\sigma\}$ is $\lambda(\sigma)=\sigma^*(L)\otimes L^{-1}$ for every $\sigma\in G$. 
Hence we have 
$$\mu^*L\otimes_{\mathcal{O}_{E\times G}}p_1^*L^{-1} \cong (1\times \lambda)^*\mathcal{P},$$
by \cite[III.13, p.\ 125, Theorem]{M}, where $\mathcal P$ is the normalized Poincar\'e line bundle 
on $E\times \hat{E}$. 

Thus $\Hom_{\mathcal{O}_\mathfrak{C}}(\varphi_*L,\varphi_*M)$ is dual to 
$$
H^1(E\times G,(1\times\lambda)^*\mathcal{P} \otimes_{\mathcal{O}_{E\times G}}p_1^*(L\otimes M^{-1})) \cong 
H^1(E\times G,(1\times\lambda)^*(\mathcal{P} \otimes_{\mathcal{O}_{E\times\hat{E}}}p_1^*(L\otimes M^{-1}))),
$$
where we abuse the notation $p_1$ to denote the first projection from both $E\times G$ and 
$E\times\hat{E}$. It follows from the Leray spectral sequence 
$H^i(G,R^jp_{2*}(1\times\lambda)^*(\mathcal{P}\otimes p_1^*(L\otimes M^{-1}))) \Rightarrow
  H^{i+j}(E\times G,(1\times\lambda)^*(\mathcal{P} \otimes p_1^*(L\otimes M^{-1})))$
that this is isomorphic to 
$$H^0(G,R^1p_{2*}(1\times\lambda)^*(\mathcal{P}\otimes p_1^*(L\otimes M^{-1}))).$$
Furthermore, we have $R^1p_{2*}(1\times\lambda)^*(\mathcal{P}\otimes p_1^*(L\otimes M^{-1})) 
\cong \lambda^*R^1p_{2*}(\mathcal{P}\otimes p_1^*(L\otimes M^{-1}))$. To see this let 
$\F= \mathcal{P}\otimes p_1^*(L\otimes M^{-1})$. Since the problem is local on $\hat{E}$, we may 
replace $\lambda \colon G \to \hat{E}$ by $\lambda \colon \Spec B \to \Spec A$ to show that 
$H^1(E_B,\F\otimes_AB) \cong H^1(E_A,\F)\otimes_AB$, where $E_A=E\times\Spec A$ and 
$E_B=E\times\Spec B$, respectively. Since $\dim E=1$ we have an open covering of $E$ consisting 
of two affine open subsets $V_1,V_2$. Then $E_A$ is covered by $U_i=V_i\times\Spec A$ with $i=1,2$ 
and $H^1(E_A,\F)$ is computed with the \v{C}ech complex 
$\mathcal{E}^{\bullet}=[0 \to \mathcal{E}^0 \to \mathcal{E}^1 \to 0]$ 
associated with $\{U_1,U_2\}$ and $\F$, i.e., there is an exact sequence 
$$\mathcal{E}^0 \to \mathcal{E}^1 \to H^1(E_A,\F) \to 0.$$
Similarly, $H^1(E_B,\F\otimes_AB)$ is computed with the \v{C}ech complex 
$\mathcal{E}^{\bullet}\otimes_AB$, i.e., 
$$\mathcal{E}^0\otimes_AB \to \mathcal{E}^1\otimes_AB \to H^1(E_B,\F\otimes_AB) \to 0$$ 
is exact. Thus the right exactness of the functor $-\otimes_AB$ leads us to the conclusion. 

Thus we see that $\Hom_{\mathcal{O}_\mathfrak{C}}(\varphi_*L,\varphi_*M)$ is dual, as a $k$-vector space, to 
$$H^0(G,\lambda^*R^1p_{2*}(\mathcal{P}\otimes p_1^*(L\otimes M^{-1}))).$$
Let $b \in \hat{E}$ be the point representing the class of $L\otimes M^{-1}$ and let 
$T_b \colon \hat{E} \to \hat{E}$ be the translation by $b$.  Then we have 
$\mathcal{P}\otimes p_1^*(L\times M^{-1}) \cong (1\times T_b)^*\mathcal{P}$ again by \cite[ibid]{M}. 
Therefore $\Hom_{\mathcal{O}_\mathfrak{C}}(\varphi_*L,\varphi_*M)$ is dual to 
$$
H^0(G,\lambda^*R^1p_{2*}(1\times T_b)^*\mathcal{P}) 
         \cong H^0(G,(T_b\circ\lambda)^*R^1p_{2*}\mathcal{P}).
$$
Since $R^1p_{2*}\mathcal{P}$ is supported at the origin $0 \in \hat{E}$ with 
$R^1p_{2*}(\mathcal{P})_0=k$ (\cite{M}, \cite[Lemma 1.1]{Od}) 
and since $T_b\circ\lambda$ is injective, $\Hom_{\mathcal{O}_\mathfrak{C}}(\varphi_*L,\varphi_*M)$ is 
one-dimensional if $T_b\circ\lambda(G)$ contains the origin $0$ of $\hat{E}$, and otherwise it 
is zero. Finally it is easy to see that $0 \in T_b\circ\lambda(G)$ if and only if $\sigma^*L \cong M$ 
for some $\sigma \in G$. 
\end{proof}

\begin{proof}[Proof of Proposition \ref{1mod3}]
By (\ref{ordinary1}) and Lemma \ref{Oda}, 
$$\varphi_*F^e_*\mathcal{O}_E \cong 
\bigoplus_{i=0}^{m-1} \omega_{\mathfrak C}^{i} \oplus \varphi_*L_1 \oplus\cdots\oplus \varphi_*L_{q-1}
$$
gives the splitting of $\varphi_*F^e_*\mathcal{O}_E$ into indecomposable bundles. 
On the other hand, since $\mathfrak C$ is $F$-split and $q=p^e \equiv 1 \pmod m$, 
$\omega_{\mathfrak C}^i$ is isomorphic to a direct summand of 
$\omega_{\mathfrak C}^i \otimes F^e_*\mathcal{O}_{\mathfrak C} 
\cong F^e_*(\omega_{\mathfrak C}^{iq}) \cong F^e_*(\omega_{\mathfrak C}^{i})$ for all integers $i$. 
Thus, comparing the splitting above with (\ref{ordinary2}), we see that $F^e_*(\omega_{\mathfrak C}^i)$ 
decomposes as 
\begin{eqnarray}\label{ordinary3}
F^e_*(\omega_{\mathfrak C}^i) \cong 
   \omega_{\mathfrak C}^i \oplus \varphi_*L_{i_1} \oplus\cdots\oplus \varphi_*L_{i_r}
\end{eqnarray}
for some $i_1,\dots,i_r$ with $1 \le i_1<\cdots<i_r \le q-1$. 
On the other hand, we have 
$$\mathcal{H}om(F^e_*(\omega_{\mathfrak C}^j),\mathcal{O}_{\mathfrak C}) 
\cong \mathcal{H}om(F^e_*(\omega_{\mathfrak C}^j),\omega_{\mathfrak C})\otimes\omega_{\mathfrak C}^{-1}
\cong F^e_*(\omega_{\mathfrak C}\otimes\omega_{\mathfrak C}^{-j})\otimes\omega_{\mathfrak C}^{-1}
\cong F^e_*(\omega_{\mathfrak C}^{-j})$$
for $j \in \mathbb{Z}$ by the adjunction formula \eqref{adjunction}. 
Hence 
$$
\mathcal{H}om(F^e_*(\omega_{\mathfrak C}^j),F^e_*(\omega_{\mathfrak C}^i))
   \cong F^e_*(\omega_{\mathfrak C}^{-j})\otimes F^e_*(\omega_{\mathfrak C}^i)
     \cong F^e_*(\omega_{\mathfrak C}^{-j}\otimes F^{e*}F^e_*(\omega_{\mathfrak C}^i)).
$$
Here 
$F^{e*}F^e_*(\omega_{\mathfrak C}^i) \cong 
F^{e*}(\omega_{\mathfrak C}^{i}) \oplus F^{e*}\varphi_*L_{i_1} \oplus\cdots\oplus F^{e*}\varphi_*L_{i_r}$ 
by (\ref{ordinary3}) and 
$F^{e*}(\omega_{\mathfrak C}^{i}) \cong \omega_{\mathfrak C}^{iq} \cong \omega_{\mathfrak C}^{i}$. 
Also, since  $\mathfrak{C}$ is a quotient stack of $E$ by $\mu_{m}$ with $p \equiv 1$ (mod $m$), the diagram \eqref{stackdiagram} for $\varphi \colon E \to \mathfrak{C}$ is the pull-back diagram by the Frobenius morphism $F$ as in Example \ref{quot}.
Since $F \colon \mathfrak{C} \to \mathfrak{C}$ is flat as we explained in Section \ref{frob}, we have $F^{e*}\varphi_*L_{i_k} \cong \varphi_*F^{e*}L_{i_k} \cong \varphi_*\mathcal{O}_E$ for each $k=1,\dots,r$. 
Thus 
$$\mathcal{H}om(F^e_*(\omega_{\mathfrak C}^j),F^e_*(\omega_{\mathfrak C}^i)) 
   \cong F^e_*(\omega_{\mathfrak C}^{i-j}) \oplus 
            F^e_*(\omega_{\mathfrak C}^{-j}\otimes\varphi_*\mathcal{O}_E)^{\oplus r}
   \cong F^e_*(\omega_{\mathfrak C}^{i-j}) \oplus F^e_*\varphi_*(\mathcal{O}_E)^{\oplus r}.
$$
It follows that the dimension of the $k$-vector space 
$\Hom(F^e_*(\omega_{\mathfrak C}^j),F^e_*(\omega_{\mathfrak C}^i))$ is equal to $r+1$ if 
$i \equiv j \pmod m$; and $r$ otherwise. This implies that the $\varphi_*L_{i_k}$'s appearing in 
decomposition (\ref{ordinary3}) are non-isomorphic to each other. The conclusion follows from 
Lemma \ref{Oda}. 
\end{proof}

Theorem \ref{genus1} now follows from Propositions \ref{2mod3} and \ref{1mod3}. 
\end{proof}

\begin{remark}
Theorem \ref{genus1} fails without the assumption that the weights are not divisible by $p$ (see 
Section 7). 
\end{remark}


\section{Stability of Frobenius push-forwards: The case where $\delta_{\mathfrak{C}} > 0$}

In this section, we assume that $\mathfrak{C}=C[\sqrt[r_{1}]{P_{1}}, \ldots, \sqrt[r_{n}]{P_{n}}]$ for a smooth projective curve $C$, 
and that $p$ does not divide any weight $r_{i}$. 
Then $\mathfrak{C}$ is Deligne-Mumford by Lemma \ref{cadman} (5), and we can define differential maps $d \colon \mathcal{O}_{\mathfrak{C}} \to \omega_{\mathfrak{C}}$ on $\mathfrak{C}$.
We put $\delta_{\mathfrak{C}} = n- \sum(1/r_i) + \deg \omega_C$ for a general smooth curve $C$.
Our goal is to  show the slope stability of Frobenius push-forward 
of line bundles on $\mathfrak{C}$ with $\delta_{\mathfrak{C}} >0$.  

As a corollary, we show that orbifold curves $\mathfrak{C}$ with $\delta_{\mathfrak{C}} > 0$ do not 
have GFFRT. This gives a negative answer to Brenner's question \cite[Question 2]{Sh} in characteristic 
$p \neq 2, 3, 7$ (see Introduction and Section 7).


\subsection{First Chern class of $F_{\ast}^{e} \mathcal{O}_{\mathfrak{C}}$}

To study slope stability of $F_{\ast}^{e} \mathcal{O}_{\mathfrak{C}}$, we compute the degree of $c_{1} (F_{\ast} \mathcal{O}_{\mathfrak{C}})$.
To this end, recall that we have 
\begin{eqnarray*}
\label{duality}
\mathcal{H} om_{\mathfrak{C}} (F_{\ast} \mathcal{O}_{\mathfrak{C}}, \omega_{\mathfrak{C}}) \cong F_{\ast} \omega_{\mathfrak{C}} 
\end{eqnarray*}
by \eqref{adjunction}.
We also consider the Frobenius push-forward of the differential map $F_{\ast}(d) \colon F_{\ast} \mathcal{O}_{\mathfrak{C}} \to F_{\ast} \omega_{\mathfrak{C}}$, which is $\mathcal{O}_{\mathfrak{C}}$-linear.
We write by $\B$ its image in $F_{\ast} \omega_{\mathfrak{C}}$.
We have a homomorphism $C^{-1} \colon \omega_{\mathfrak{C}} \to F_{\ast} \omega_{\mathfrak{C}} / \B$ from the similar arguments as in \cite[9.14]{EV}.
Since the characteristic $p$ does not divide any weight $r_{i}$, this is an isomorphism.
The inverse $C \colon F_{\ast} \omega_{\mathfrak{C}}/ \B \cong \omega_{\mathfrak{C}}$ is called the \emph{Cartier operator}.

We have exact sequences
\begin{eqnarray*}
0 \to \mathcal{O}_{\mathfrak{C}} \stackrel{F}{\to} F_{\ast} \mathcal{O}_{\mathfrak{C}} \to \B \to 0  \\
0 \to \B \to F_{\ast} \omega_{\mathfrak{C}} \stackrel{C}{\to} \omega_{\mathfrak{C}} \to 0. 
\end{eqnarray*}
By these facts, we have
$$\det F_{\ast} \mathcal{O}_{\mathfrak{C}} \cong \det \B \cong \det F_{\ast}(\omega_{\mathfrak{C}}) \otimes \omega_{\mathfrak{C}} ^{-1} 
                               \cong \omega_{\mathfrak{C}}^{p} \otimes (\det F_{\ast} \mathcal{O}_{\mathfrak{C}})^{-1} \otimes \omega_{\mathfrak{C}} ^{-1}. 
$$
Hence we have
\begin{eqnarray}
\label{chern1}
c_{1}(F_{\ast} \mathcal{O}_{\mathfrak{C}}) = \frac{p-1}{2} K_{\mathfrak{C}}
\end{eqnarray}
in the Chow ring $A^{\bullet}(\mathfrak{C})$ of $\mathfrak{C}$.

\begin{proposition}
\label{chern}
For any vector bundle $\E$ of rank $r$ on $\mathfrak{C}$, we have
$$
c_{1}(F_{\ast} \E) = \frac{p-1}{2} r K_{\mathfrak{C}} + c_{1}(\E)
$$
in $A^{\bullet}(\mathfrak{C})$.

\end{proposition}
\proof
We have a full flag of sub-bundles
$$
0 \subset \E_{1} \subset \E_{2} \subset \cdots \subset \E_{r} = \E
$$
corresponding to a section of the full flag bundle of $\E$ over an orbifold curve $\mathfrak{C}$.
Hence it is enough to prove for a line bundle.

We take a line bundle $\mathcal{O}_{\mathfrak{C}}(\sum s_{i} Q_{i}^{+} - \sum t_{j} Q_{j}^{-})$, where $s_{i} , t_{j} >0$, and $Q_{i}^{+}, Q_{j}^{-}$ are closed points on $\mathfrak{C}$. 
We show  the assertion by induction on $\sum_{i} s_{i} + \sum_{j} t_{j}$.
The first step for the induction follows from \eqref{chern1}.
For the inductive step, it is enough to show that for any line bundle $\mathcal{L}$ and any closed point $Q$ on $\mathfrak{C}$, the assertions for $\mathcal{L}$ and $\mathcal{L}(-Q)$ are equivalent.
To this end, we introduce an exact sequence
\begin{eqnarray}
\label{6.2}
0 \to \det F_{\ast} \mathcal{L}(-Q) \to \det F_{\ast} \mathcal{L} \to k(Q) \otimes \rho \to 0,
\end{eqnarray}
where $\rho$ is a one-dimensional representation of the automorphisms group of a closed point $Q \in \mathfrak{C}(\Spec k)$.
Note here that if $Q$ is a stacky point, then the inclusion $\mathcal{L}(-Q) \to \mathcal{L}$ locally corresponds to the ideal generated by the coordinate vanishing at $Q$, and this is an element of degree one around $Q$.
Thus we need the twist by $\rho$ in \eqref{6.2}. 
Now the inductive step follows from $c_{1}( F_{\ast} \mathcal{L}) = c_{1}( F_{\ast} \mathcal{L}(-Q)) + c_{1}(\mathcal{O}_{\mathfrak{C}}(Q))$ due to the exact sequence \eqref{6.2}.
\endproof


\subsection{Slope stability}
For a vector bundle $\E$ on $\mathfrak{C}$, we define the \emph{slope} $\mu(\E)$ of $\E$ by
$$
\mu(\E) = \frac{\deg c_{1}(\E)}{\rk \E}.
$$
As in the previous subsection, we assume that $p$ does not divide any weight $r_{i}$.
\begin{proposition}
\label{rr}
We have 
$$
\mu(F^{\ast} F_{\ast} \E) = \mu(\E) + \frac{p-1}{2} \delta_{\mathfrak{C}}.
$$
\end{proposition}
\proof
By Proposition \ref{chern}, we have $\mu(F_{\ast} \E) = p^{-1} (\mu(\E) + \frac{p-1}{2} \delta_{\mathfrak{C}})$.
Since $\mu(F^{\ast} \F) =  p \mu(\F)$ for any vector bundle $\F$ on $\mathfrak{C}$, the assertion follows.
\endproof

\begin{definition}
We say that a vector bundle $\E$ on $\mathfrak{C}$ is \emph{semi-stable} if for any non-trivial proper 
sub-bundle $\E'$ of $\E$, we have an inequality
$$
\mu(\E') \le \mu(\E).
$$
If the inequality is always strict, we say that $\E$ is \emph{stable}. 
\end{definition}
Semi-stability or stability is equivalent to the condition that the same inequality holds for any non-trivial subsheaf of $\E$ as in \cite[1.2.2]{OSS}.
We also remark that it is different from the stability defined in \cite{N}. But it is enough for our purpose to 
show the indecomposability of Frobenius push-forwards $F^{e}_{\ast} \mathcal{O}_{\mathfrak{C}}$ for $e>0$. For this purpose, we 
follow the arguments in \cite{KS}, \cite{Su}. It is straightforward to modify their arguments to our situation.
The only difference is that we must consider grading even in local situation.

For a vector bundle $\E$ on $\mathfrak{C}$, there exists a connection 
$$\nabla = \id \otimes d \colon F^{\ast} \E = F^{-1} \E \otimes_{F^{-1} \mathcal{O}_{\mathfrak{C}}} \mathcal{O}_{\mathfrak{C}} \to F^{\ast} \E \otimes \Omega_{\mathfrak{C}} = F^{-1} \E \otimes_{F^{-1} \mathcal{O}_{\mathfrak{C}}} \Omega_{\mathfrak{C}}
$$ 
called the \emph{canonical connection} similarly for a variety over $k$ as in \cite{KS}, \cite{Su}.
Here $d \colon \mathcal{O}_{\mathfrak{C}} \to \Omega_{\mathfrak{C}}$ is the differential, which is $F^{-1}\mathcal{O}_{\mathfrak{C}}$-linear. 
This is locally written as
$$
M \otimes_{B} B \to M \otimes_{B} B \otimes_{B} \Omega_{B/k} \cong M \otimes_{B} \Omega_{B/k};\; m \otimes f \mapsto m \otimes df,
$$ 
where $B$ has a grading from the local description \eqref{local}, and $\nabla$ preserves the grading.

When $\E=F_{\ast} W$ for a vector bundle $W$ on $\mathfrak{C}$, we introduce the canonical filtration of $\F=F^{\ast} F_{\ast} W$ due to \cite{KS}, \cite{Su}.
We put $\F_{0} = \F$, $\F_{1} = \ker (F^{\ast} F_{\ast} W \to W)$, and $\F_{\ell} = \ker (\F_{\ell -1} \stackrel{\nabla}{\to} \F \otimes \Omega_{\mathfrak{C}} \to (\F /\F_{\ell -1}) \otimes \Omega_{\mathfrak{C}})$.
We have the filtration:
$$
\cdots \subset \F_{3} \subset \F_{2} \subset \F_{1} \subset \F_{0} : = \F := F^{\ast} F_{\ast} W.
$$

\begin{definition}
We call this filtration $\F_{\bullet}$ the \emph{canonical filtration} on $\F=F^{\ast} F_{\ast} W$.
\end{definition}


Since the local computations in \cite{KS}, \cite{Su} holds equivariantly, we have the following lemma.

\begin{lemma}\label{filt}
$\phantom{Under the above notation we have the following}$
\begin{enumerate}
\item[$(1)$] $\F_{0} / \F_{1} \cong W$, $\nabla (\F_{\ell +1}) \subset \F_{\ell} \otimes \Omega_{\mathfrak{C}}$ for $\ell \ge 1$.
\item[$(2)$] $\nabla$ induces an isomorphism $\F_{\ell}/\F_{\ell +1} \cong \F_{\ell -1} / \F_{\ell} \otimes \Omega_{\mathfrak{C}}$ for $\ell < p$.
\item[$(3)$] If $W$ is stable $($resp.\ semi-stable$)$, then $\F_{\ell} / \F_{\ell +1}$ is stable $(resp.\ 
$semi-stable$)$ for any $\ell$.
\end{enumerate}
\end{lemma}
\proof
(1) follows from the Definition.
(2) follows from \cite[Lemma 2.1 (ii)]{Su} and the fact that the isomorphism preserves grading.
(3) follows from (1) and (2).
\endproof
By this lemma, we have $\F_{p}=0$ and $\F_{p-1} \neq 0$.

\begin{theorem}\label{stability}
Let $W$ be a vector bundle on an orbifold curve $\mathfrak{C}$, and suppose 
that $p$ does not divide any weight $r_i$. If $\delta_{\mathfrak{C}}>0$ and $W$ is stable $($resp.\ $\delta_{\mathfrak{C}} \ge 0$ and $W$ 
is semi-stable$)$, then $F_{\ast} W$ is stable $($resp.\ semi-stable$)$. 
\end{theorem}

\proof
It follows from the similar argument as in the proof of \cite[Theorem 2.2]{Su}.
But we give a proof for the convenience of readers.

We take a non-trivial sub-bundle $\E' \subset F_{\ast} W$, and show $\mu(\E') < \mu(F_{\ast} W)$ $($resp.\ $\mu(\E') \le \mu(F_{\ast} W))$.
By Proposition \ref{rr}, we have 
\begin{eqnarray}
\label{slope1}
\mu(F^{\ast} F_{\ast} W) = \mu(W) + \frac{p-1}{2} \delta_{\mathfrak{C}}.
\end{eqnarray}
We consider the induced filtration
$$
0 \subset \F^{m} \cap F^{\ast} \E' \subset \cdots \subset \F^{1} \cap F^{\ast} \E' \subset F^{\ast} \E',
$$
where we assume $\F^{m+1} \cap F^{\ast} \E' =0$ and $\F^{m} \cap F^{\ast} \E' \neq 0$ for $m <p$.
If we put 
$$
r_{\ell}=\rk\left( \frac{\F^{\ell} \cap F^{\ast} \E' }{ \F^{\ell+1} \cap F^{\ast} \E' }\right),
$$ 
then we have
\begin{eqnarray}
\label{slope2}
\mu(F^{\ast} \E') = 
\frac1{\rk (\E')} \sum_{\ell =0}^m \mu \left(\frac{\F^{\ell} \cap F^*\E'}{\F^{\ell+1} \cap F^*\E'}\right) r_{\ell} 
 \le \frac{1}{\rk (\E')} \sum_{\ell=0}^{m} (\mu(W) + \ell \delta_{\mathfrak{C}}) r_{\ell},
\end{eqnarray}
where the last inequality follows from Lemma \ref{filt} (2), (3).

Putting \eqref{slope1} and \eqref{slope2} together, we have
\begin{eqnarray}
\nonumber
\mu(F_{\ast} W) - \mu(\E') 
&=& \frac{1}{p} (\mu(F^{\ast} F_{\ast} W) - \mu(F^{\ast} \E')) \\
\label{sum0}
&\ge& \frac{\delta_{\mathfrak{C}}}{\rk(\E') p} \sum_{\ell=0}^{m} \left(\frac{p-1}{2} - \ell \right) r_{\ell}.
\end{eqnarray}
If $m \le \frac{p-1}{2}$, then the last sum in \eqref{sum0} is greater than $0$, and we get the desired inequality.
Hence we may assume $m > \frac{p-1}{2}$.

Then the last sum in \eqref{sum0} is equal to
\begin{eqnarray}
\label{sum}
\sum_{\ell=m+1}^{p-1} \left(\ell - \frac{p-1}{2}\right) r_{p-1- \ell} + \sum_{\ell 
 > \frac{p-1}{2}}^{m} \left(\ell - \frac{p-1}{2}\right) \left(r_{p-1-\ell} - r_{\ell}\right),
\end{eqnarray}
since we have $\ell - \frac{p-1}{2}=\frac{p-1}{2} - (p -1 - \ell)$, and $p -1-\ell$ runs from $0$ to $p-m-2$ in the first sum, and $p-m-1 \le p-1-\ell  < \frac{p-1}{2}$ in the second sum.
On the other hand, the isomorphism in Lemma \ref{filt} (2) induces an inclusion $\F^{\ell}\cap F^*\E' / \F^{\ell+1}\cap F^*\E' \subset \left( \F^{\ell-1}\cap F^*\E' / \F^{\ell}\cap F^* \E' \right) \otimes \Omega_{\mathfrak{C}}$, since we have $\nabla(F^{\ast} \E') \subset F^{\ast} \E' \otimes \Omega_{\mathfrak{C}}$ by the definition $\nabla = \id_{F^{-1} F_{\ast} W} \otimes d$.
Hence we have 
$$
r_{0} \ge r_{1} \ge \cdots \ge r_{m},
$$
and \eqref{sum} is greater than, or equal to $0$. 
This implies the semi-stability of $F_*W$. 

Finally we assume that $\delta_{\mathfrak{C}}>0$ and $W$ is stable.
If $\mu(F_{\ast} W) - \mu(\E')=0$, then we have an equality in \eqref{slope2}. 
This implies 
$\F^{\ell} \cap F^{\ast} \E' / \F^{\ell+1} \cap F^{\ast} \E'  = \F^{\ell} / \F^{\ell-1}$ 
for all $\ell=0,1,\ldots,m$, and we have $r_{0} = r_{1} = \cdots = r_{m} = \rk(W)$.
Furthermore since \eqref{sum} must be equal to $0$, we have $m=p-1$.
This implies $\rk (\E') = \rk F_{\ast}W$, and a contradiction.
\endproof

As a direct corollary, we have the following theorem.

\begin{theorem}\label{g>1}
We assume that $\delta_{\mathfrak{C}}$ is greater than $0$.
Then for any $e$, the Frobenius push-forward $F^{e}_{\ast} \mathcal{O}_{\mathfrak{C}}$ is indecomposable.
\end{theorem}

\proof
For a contradiction, we assume that $F^{e}_{\ast} \mathcal{O}_{\mathfrak{C}}$ is decomposed into non-trivial vector bundles 
$\E_{1}$ and $\E_{2}$. Then both $\E_{1}$ and $\E_{2}$ are sub-bundles of $F^{e}_{\ast} \mathcal{O}_{\mathfrak{C}}$. 
Hence we get inequalities 
$\mu(\E_1) < \mu(F^e_{\ast}\mathcal{O}_\mathfrak{C}) < \mu(\E_2)$ and $\mu(\E_2) < \mu(F^e_{\ast}\mathcal{O}_\mathfrak{C}) < \mu(\E_1)$, 
and a contradiction.
\endproof


\section{The FFRT property of $R(\PP^1,D)$ and concluding remarks}

In this section we prove our results on the FFRT property of $R(C,D)$. 

\begin{proposition}
Let $C$ be a smooth projective curve of genus $g \ge 1$ over an algebraically closed field of characteristic $p>0$ and $D$ an ample $\Q$-Cartier divisor on $C$. Then the graded ring $R(C,D)$ does not have FFRT. 
\end{proposition}

\begin{proof}
First note that $C$ does not have GFFRT. This follows from Lemma \ref{atiyah} when $g=1$ and \cite{Su} when case $g>1$. Now let $\pi \colon \mathfrak{C} \to C$ be the orbifold curve constructed with respect to the fractional 
part of $D$ and let $L=\mathcal{O}_\mathfrak{C}(\pi^*D)$. Since $\pi_*F^e_*\mathcal{O}_\mathfrak{C} \cong F^e_*\mathcal{O}_C$, it follows that $\mathfrak{C}$ does 
not have GFFRT as well. Then $(\mathfrak{C},L)$ does not have GFFRT, and the result follows from Corollary 
\ref{FFRT-GFFRT}. 
\end{proof}

It follows from the proposition above that $R(C,D)$ has FFRT only if $C \cong \PP^1$. 
To state our main theorem let us fix the notation used through the remainder of this paper. 
Let $R=R(\PP^1,D)$ be a two-dimensional normal graded ring with $R_0=k$ an algebraically closed field 
of characteristic $p>0$.
Here  
$$D=\sum_{i=1}^n \frac{s_i}{r_i} P_i$$ 
is an ample $\Q$-divisor on $\PP^1$, where $P_1,\dots,P_n$ are distinct closed points on $\PP^1$, and $r_i>0$ and $s_i$ are coprime integers. 

Let $\mathfrak{C}=\PP^1[\sqrt[r_1]{P_1},\dots,\sqrt[r_n]{P_n}\,]$ be the weighted projective line with weight 
$(r_1,\dots,r_n)$. The following are well-known; see e.g., \cite{H} and references therein. 
\begin{enumerate}
\item[(1)] $R$ has a log terminal singularity if and only if $\delta_{\mathfrak{C}}=\deg \omega_\mathfrak{C}= - 2 + \sum_{i=1}^n\frac{r_i-1}{r_i}<0$.
\item[(2)] $R$ has a log canonical singularity if and only if $\delta_{\mathfrak{C}} \le0$.
\end{enumerate}
In the case of (1) above, 
$R$ has FFRT by Theorem \ref{FRT}. 
On the other hand, we have the following theorem.

\begin{theorem}\label{MainThm}
In the notation as above, suppose that $\delta_{\mathfrak{C}} \ge 0$ and that $p$ does not divide any $r_i$. 
Then $R=R(\PP^1,D)$ does not have FFRT. 
\end{theorem}

\proof
It follows from Theorems \ref{genus1} and \ref{g>1} that $\mathfrak{C}$ does not have GFFRT. Then for 
$L=\mathcal{O}_\mathfrak{C}(\pi^*D)$, the pair $(\mathfrak{C},L)$ does not have GFFRT, and the result 
follows from Corollary \ref{FFRT-GFFRT}. 
\endproof

In Theorem \ref{MainThm} the assumption that $p$ does not divide any $r_i$ is really necessary as we 
will see in the following examples. 

\begin{proposition}
\label{propsand}
Assume that all $r_{i}$ are equal to $p$, then $R=R(\PP^{1}, D)$ has FFRT.
\end{proposition}
\begin{proof}
Under the assumption, $\mathfrak{C}$ is a Frobenius sandwich as in Example \ref{sand}, that is, the Frobenius morphism of $\PP^1$ is factorized as  
$$
F_{\PP^1} \colon \PP^1 \stackrel{\varphi}{\lra} \mathfrak{C} \stackrel{\pi}{\lra} \PP^1
$$
and the Frobenius $F=F_\mathfrak{C}$ of $\mathfrak{C}$ is factorized as $F=\varphi\circ\pi$. 
Then for $e\ge 1$, the $e$-th Frobenius on $\mathfrak{C}$ is factorized as $F^e=\varphi\circ(F_{\PP^1})^{e-1}\circ\pi$. 
Thus the claim is reduced to the following splitting ( see \cite[Theorem 4.5]{OU} for example)
$$
(F_{\PP^1})^{e}_*\mathcal{O}_{\PP^1}(a)=\bigoplus_{k=0}^{p^{e}-1} \mathcal{O}_{\PP^1}\left(\lfloor \frac{a -k}{p^{e}} \rfloor\right).
$$

In fact, for a line bundle $L$ on $\mathfrak{C}$ of $\deg L>0$ and $0 \le i \le p^e-1$, we have  
$$
F^e_*(L^i) = \varphi_*(F_{\PP^1})^{e-1}_*\pi_*(L^i) = \varphi_*(F_{\PP^1})^{e-1}_*\mathcal{O}_{\PP^1}(a_i),
$$ 
where $1-n \le a_i \le (p^e-1)\deg L$. 
Hence if $p^{e-1} \ge \max \lbrace n, \deg L \rbrace$, then 
$(F_{\PP^1})^{e-1}_*\mathcal{O}_{\PP^1}(a_i)$
splits into line bundles belonging to
$$
\left\lbrace \mathcal{O}_{\PP^1}(-2),\dots,\mathcal{O}_{\PP^1}(p\deg L-1) \right \rbrace, 
$$
so that, $F^e_*(L^i)$, with $0\le i\le p^e-1$ and $e\gg 1$, splits into finitely many vector bundles belonging to 
$$
\lbrace 
\varphi_*\mathcal{O}_{\PP^1}(-2),\varphi_*\mathcal{O}_{\PP^1}(-1),\dots,\varphi_*\mathcal{O}_{\PP^1}(p\deg L-1)
\rbrace.
$$ 

We therefore conclude that $(\mathfrak{C},L)$ has GFFRT. 
\end{proof}

\begin{example}\label{eg-237}
Let $R=k[x,y,z]/\langle x^2+y^3+z^7\rangle$. This is not a rational singularity but $\Proj R \cong \PP^1$ 
and $R \cong R(\PP^1,D)$ for a $\Q$-divisor $D=\frac12(\infty)-\frac13(0)-\frac17(1)$ on $\PP^1$. 
By Theorem \ref{MainThm}, $R$ does not have FFRT if $p\ne 2,3,7$. On the other hand, $R$ has FFRT 
if $p=2,3,7$ (\cite{Sh}). 
\end{example}

We also have an example of rational singularity which does not have FFRT.
\begin{example}\label{eg-333}
Let $R=R(\PP^1,D)$ for a $\Q$-divisor $D=\frac13(\infty)+\frac13(0)-\frac13(1)$ on $\PP^1$. 
This is a rational log canonical singularity but not log terminal. 
By Theorem \ref{MainThm}, $R$ does not have FFRT if $p \ne 3$.
\end{example}

\begin{remark}
In Examples \ref{eg-237} and \ref{eg-333}, the ring $R=R(\PP^1,D)$ does not have finite representation type 
in any characteristic $p>0$, since $\delta_\mathfrak{C}\ge 0$. However, $R$ has FFRT in exceptional characteristics, 
that is, $p=2,3,7$ in Example \ref{eg-237} and $p=3$ in Example \ref{eg-333}. In these exceptional cases 
$\mathfrak{C}$ turns out to be a Frobenius sandwich, as we have seen in the proof of Proposition \ref{propsand}. 
\end{remark}

In the two-dimensional case, it is known that $F$-regular rings have log terminal singularities.
Hence $F$-regular implies FFRT property as we saw in Section 4.
However, this statement cannot hold in the higher dimensional case.
For there exists an example which is $F$-regular and does not have FFRT property \cite{SS}, \cite[Remark 3.4. (2)]{TT}.

\begin{question}
Let $\mathfrak{X}$ be a root stack in arbitrary dimensions, and $L$ a line bundle on $\mathfrak{X}$. 
Is there any difference between the GFFRT properties of $\mathfrak{X}$ and the pair $(\mathfrak{X},L)$?
Here the latter property is equivalent to the FFRT property of the section ring $R=R(\mathfrak{X}, L)$ (cf. Corollary \ref{FFRT-GFFRT}).
\end{question}

\end{document}